\documentclass[11pt,reqno]{amsart}

\usepackage{soul}
\usepackage{afterpage}
\usepackage{graphicx}
\usepackage{cancel}
\usepackage{bm}
\usepackage{amssymb}
\usepackage{esint}
\usepackage{epstopdf}
\usepackage{color}

\usepackage{enumitem}
\usepackage[colorlinks = true,
            linkcolor = blue,
            urlcolor  = blue,
            citecolor = blue,
            anchorcolor = blue]{hyperref}
            
\usepackage{url}
\usepackage[top=1in, left=1in, right=1in, bottom=1in]{geometry}

\let\oldtocsection=\tocsection
\let\oldtocsubsection=\tocsubsection
\let\oldtocsubsubsection=\tocsubsubsection

\renewcommand{\tocsection}[2]{\hspace{0em}\oldtocsection{#1}{#2}}
\renewcommand{\tocsubsection}[2]{\hspace{2em}\oldtocsubsection{#1}{#2}}
\renewcommand{\tocsubsubsection}[2]{\hspace{4em}\oldtocsubsubsection{#1}{#2}}

\newtheorem{theorem}{Theorem}[section]
\newtheorem{lemma}[theorem]{Lemma}

\newtheorem{corollary}[theorem]{Corollary}

\newtheorem{remark}[theorem]{Remark}
\newtheorem{definition}[theorem]{Definition}

\AtBeginDocument{

}

\def\R{{\mathbb R}}

\def\N{{\mathbb N}}

\def\cC{{\mathcal C}}

\def\cL{{\mathcal L}}

\def\a{\alpha}
\def\b{\beta}
\def\e{\varepsilon}
\def\eps{\varepsilon}
\def\d{\delta}
\def\D{\Delta}

\def\l{\lambda}
\def\L{\Lambda}
\def\m{\mu}

\def\p{\partial}
\def\r{\rho}

\def\t{\tau}
\def\w{\omega}
\def\W{\Omega}

\def\1{\left(}
\def\2{\right)}
\def\3{\left\{}
\def\4{\right\}}
\def\8{\infty}
\def\sm{\setminus}
\def\ss{\subseteq}

\DeclareMathOperator*{\diam}{diam}

\DeclareMathOperator*{\osc}{osc}
\DeclareMathOperator*{\dist}{dist}

\usepackage[hypcap=false]{caption}
\usepackage{subcaption}

\usepackage[dvipsnames]{xcolor}

\def\sideremark#1{\ifvmode\leavevmode\fi\vadjust{\vbox to0pt{\vss
 \hbox to 0pt{\hskip\hsize\hskip1em
 \vbox{\hsize2.1cm\tiny\raggedright\pretolerance10000
  \noindent #1\hfill}\hss}\vbox to15pt{\vfil}\vss}}}%

\makeatletter
\def\namedlabel#1#2{\begingroup
    #2%
    \def\@currentlabel{#2}%
    \phantomsection\label{#1}\endgroup
}
\makeatother

\usepackage[cm]{fullpage}

\title{Classical solutions to integral equations with zero order kernels}

\author[H. A. Chang-Lara]{H\'ector A. Chang-Lara}
\thanks{H. A. Chang-Lara. Centro de Investigación en Matemáticas. Guanajuato, México.}
\address{
H\'ector A. Chang-Lara,\newline
Matemáticas Básicas,\newline
Centro de Investigación en Matemáticas,\newline
Jalisco S/N, Col. Valenciana,\newline
C.P. 36000, Guanajuato, Mexico}
\email{hector.chang@cimat.mx}

\author[A. Saldaña]{Alberto Saldaña}
\thanks{A. Saldaña. Instituto de Matem\'aticas, Universidad Nacional Aut\'onoma de M\'exico. Ciudad de México, México.}
\address{
Alberto Saldaña,\newline
Instituto de Matem\'aticas,\newline 
Universidad Nacional Aut\'onoma de M\'exico,\newline 
Circuito Exterior, Ciudad Universitaria,\newline
C.P. 04510, Mexico City, Mexico.}
\email{alberto.saldana@im.unam.mx}


\begin{document}

\begin{abstract}
We show global and interior higher-order log-Hölder regularity estimates for solutions of Dirichlet integral equations where the operator has a nonintegrable kernel with a singularity at the origin that is weaker than that of any fractional Laplacian. As a consequence, under mild regularity assumptions on the right hand side, we show the existence of classical solutions of
Dirichlet problems involving the logarithmic Laplacian and the logarithmic Schrödinger operator.
\end{abstract}

\maketitle


\section{Introduction}\label{sec:intro}

In this paper we consider integrodifferential operators of the form
\begin{align}\label{Lk}
L_K u(x):=\int_{B_1(x)}\frac{u(x)-u(y)}{|x-y|^{N}}K(x,y-x)\, dy,
\end{align}
where $N\geq 1$ is the dimension and $K$ is a nonnegative bounded function and bounded away from zero.  Here $K$ represents the presence of some heterogeneity and anisotropy in the (anomalous) diffusion. Note that $L_K$ has a finite range of interaction (because the integral is over the set $B_1(x)$); however, as can be seen in the two examples we present below, the regularity theory for $L_K$ can be applied to more general zero order singular integral operators with an infinite range of interaction. Observe that the kernel $|x-y|^{-N}$ is the borderline case for hypersingular integrals and hence we call it a \emph{zero order kernel}.

Operators of the type \eqref{Lk} are related to geometric stable Lévy processes, we refer to \cite{bass94,beghin14,feulefack2021logarithmic,FKT20,MR3626549,kozubowski99,vsikic06} and the references therein for an overview of the different applications (in engineering, finances, physics, etc) and of the theoretical relevance of these operators in probability and in potential theory.

In this paper we study the regularity of solutions to Dirichlet equations involving $L_K$.  Since the singularity of the kernel in $L_K$ is weaker than that of any fractional Laplacian, the corresponding solutions  might not belong to any Hölder space.  As a consequence, the regularity of solutions has to be analyzed in generalized Hölder spaces with logarithmic weights. 

Before we present our general results, we focus first on a couple of paradigmatic particular cases, which are also the main motivation for this research.

\subsection{The logarithmic Laplacian and the logarithmic Schrödinger equation}

The \emph{logarithmic Laplacian}, denoted by $L_\Delta$, is a pseudodifferential operator with Fourier symbol $2\ln|\xi|$.  It appears naturally as a first order expansion of the fractional Laplacian $(-\Delta)^s$ as $s\to 0^+$; in particular, in \cite[Theorem 1.1]{MR3995092} it is shown that, for all $\varphi\in C^\infty_c(\R^N)$,
\begin{align*}
(-\Delta)^s\varphi = \varphi + sL_\Delta \varphi + o(s)\qquad \text{as $s\to 0^+$ in }L^p(\R^N) \text{ with }1<p\leq \infty,
\end{align*}
 where $(-\Delta)^s$ denotes the pseudodifferential operator with Fourier symbol $|\xi|^{2s}$.  The operator $L_\Delta$ can be written as the singular integral
 \begin{align}\label{LL}
     L_\Delta u(x)=c_N\int_{\R^N}\frac{u(x)\chi_{B_1(x)}(y)-u(y)}{|x-y|^N}\, dy+\rho_N u(x),
 \end{align}
 where $c_N$ and $\rho_N$ are explicit constants given in \eqref{constants} below. In \cite[Proposition 1.3]{MR3995092} it is proved that \eqref{LL} is well defined under rather mild regularity assumptions on $u$, for instance, if $u$ has compact support and it is Dini continuous at $x$.
 
The study of Dirichlet problems involving $L_\Delta$ is a key element in the study of fractional equations as the order $s$ goes to zero. This asymptotic analysis is relevant, both in applications and for theoretical reasons.  In applications, for instance, several phenomena that are modeled with a fractional-type diffusion are optimized in some sense for small values of $s$.  On the other hand,  the mathematical structures that appear in the limit as $s\to 0^+$ for linear and nonlinear fractional problems are highly nontrivial and interesting. For brevity, here we simply refer to \cite{Caffarelli17,MR3995092,hernandez2022small,MR4141492,MR4279386} for an overview of these results and for further references. 
 
 The research of boundary value problems in an open bounded set $\Omega\subset \R^N$ such as 
 \begin{align}\label{ll:1}
    L_\Delta u =f \text{ in }\Omega,\qquad
    u=0\text{ in }\R^N\backslash \Omega
 \end{align}
 (for some adequate $f$), is very recent.   In \cite{MR3995092} it is shown that \eqref{ll:1} has a variational structure and weak solutions can be found with standard arguments whenever $\Omega$ has a small measure. This restriction on $\Omega$ guarantees that the operator $L_\Delta$ has a maximum principle; in particular, in some large open sets it can happen that \eqref{ll:1} does not have a solution. Furthermore, whenever there is a weak solution, it is not trivial to find conditions on $f$ to guarantee that $u$ is a \emph{classical solution}, namely, that $L_\Delta u$ can be computed pointwisely in $\Omega$.  The only previous higher-order regularity result that we are aware of is \cite{feulefack2021nonlocal}, which shows (among other results) that a weak solution $u$ belongs to $C^\infty$ if $f\in C^\infty$.   The main difficulty is that the kernel involved in \eqref{LL} is the borderline case for hypersingular integrals and thus its regularizing properties are weak and nontrivial to characterize; in particular, solutions might not belong (even locally) to the Hölder space $C^\a$ for any $\a\in(0,1)$. In other words, the question that we aim to answer here is the following:
  \begin{center}
   \emph{Under which conditions on $f$ do we have a classical solution $u$ of \eqref{ll:1} and what is the improvement in the regularity of $u$  with respect to the smoothness of $f$?}
  \end{center}
  
 To answer this question, we need to consider \emph{generalized Hölder spaces} with a (truncated logarithmic) weight
 \begin{align*}
 \ell(\r) := \frac{1}{|\ln(\min\{\r,\frac{1}{10}\})|}.     
 \end{align*}
 This function is concave in $(0,\infty)$ and behaves like $|\ln \rho|^{-1}$ as $\rho\to 0^+$, see Figure \ref{f}.

\setlength{\unitlength}{1cm}
\begin{figure}[ht] 
\begin{center}
\begin{picture}(5,3.5)
\put(0,0){\includegraphics[width=5cm]{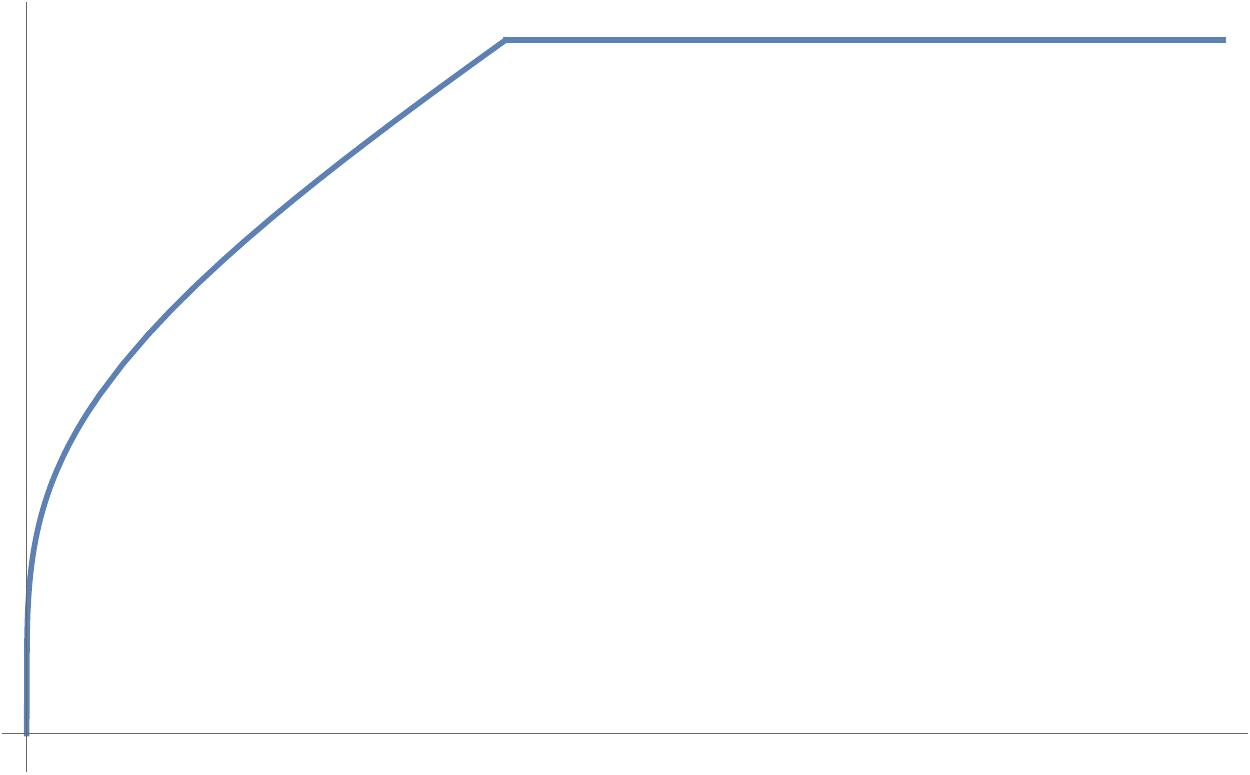}}
\put(-0.6,3){$\ell(r)$}
\put(5.2,0){$r$}
\put(1.8,0.6){$\frac{1}{10}$}
\put(1.95,0.1){$|$}
\end{picture}
\end{center}
 \caption{The function $\ell$.}\label{f}
\end{figure}
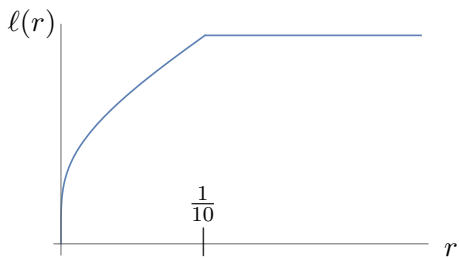

 Next we introduce some notation. Given $\alpha\in(0,1)$ and $N\geq 1$, let
 \begin{align*}
 \mathcal X^\alpha(\Omega) &:= \{u:\R^N\to\R\ | \ \|u\|_{\mathcal X^\alpha(\Omega)}<\8 \text{ and }u=0\text{ in }\R^N\backslash \W\},\\ 
 \mathcal Y(\Omega) &:= \{f:\W\to\R \ | \ \|f\|_{\mathcal Y(\Omega)}<\8\},
 \end{align*}
 where
 \begin{align*}
\|u\|_{\mathcal X^\alpha(\Omega)} &:=
\sup_{\substack{x,y\in \R^N\\x\neq y}}\frac{|u(x)-u(y)|}{\ell^\a(|x-y|)}
+\sup_{\substack{x,y\in \W\\x\neq y}}\ell^{1+\alpha}(d(x,y))\frac{|u(x)-u(y)|}{\ell^{1+\a}(|x-y|)},\\
\|f\|_{\mathcal Y(\Omega)} &:= \|f\|_{L^\infty(\W)}
+\sup_{\substack{x,y\in \W\\x\neq y}}\ell^{2}(d(x,y))\frac{|f(x)-f(y)|}{\ell(|x-y|)},\\
d(x)&:=\dist(x,\partial \Omega),\qquad d(x,y):=\min(d(x),d(y)).
\end{align*}
The second term in the norm $\|\cdot\|_{\mathcal X^\alpha}$ accounts for higher-order regularity in the interior of $\Omega$ (which might deteriorate close to the boundary $\partial \Omega$); this regularity is enough to compute pointwisely the operator $L_\Delta u$ in the interior of $\Omega$. Moreover, the term
\begin{align}\label{es:int}
\sup_{\substack{x,y\in \R^N\\x\neq y}}\frac{|u(x)-u(y)|}{\ell^\a(|x-y|)}
\end{align}
in the norm $\|\cdot\|_{\mathcal X^\alpha}$ relates to global lower-order regularity in $\R^N$, which in particular describes the behavior of $u$ across the boundary $\partial \Omega$.  We emphasize that, since $\alpha\in(0,1)$, having only that \eqref{es:int} is finite is not enough to compute $L_\Delta u$ pointwisely, and this is why we call it lower order regularity.  A discussion of these norms and the associated Banach spaces can be found in Section \ref{sec2.1}.

Our main regularity contribution to \eqref{ll:1} is the following Fredholm-Alternative-type result. 
\begin{theorem}\label{loglap:cor}
Let $\W\ss\R^N$ be an open bounded Lipschitz set satisfying a uniform exterior ball condition.  Then there is $\alpha=\alpha(N,\W)\in(0,1)$ such that exactly \underline{one} of the following alternatives holds:
\begin{enumerate}
    \item For every $f \in \mathcal Y(\Omega)$ there exists a unique \underline{classical} solution $u\in \mathcal X^\alpha(\Omega)$ of 
    \begin{align*}
L_\Delta u=f \text{ in } \W,\qquad
u=0 \text{ on } \R^N\sm\W.
\end{align*}
Moreover, $\|u\|_{\mathcal X^\alpha(\Omega)} \leq C\|f\|_{\mathcal Y(\Omega)}$ for some constant $C=C(N,\Omega)>0$.
\medskip
\item There is a non-trivial classical solution $u\in \mathcal X^\alpha(\Omega)$ of $L_\Delta u=0$.
\end{enumerate}
\end{theorem}

Observe that the solution $u$ has an improved interior regularity (with respect to that of $f$) which allows the pointwise evaluation of $L_\Delta u$ in $\Omega$ (see Theorem \ref{prop3} and Lemma \ref{T:comp}).

The second alternative in Theorem \ref{loglap:cor} can be discarded whenever $L_\Delta$ satisfies a maximum principle, which is equivalent to having a strictly positive principal eigenvalue. This is the case under the geometric hypotheses given in \cite[Corollary 1.9]{MR3995092}.  On the other hand,  the principal eigenvalue of $L_\Delta$ (denoted $\lambda_1^L$) is bounded from above by the logarithm of the principal eigenvalue of the Dirichlet Laplacian on $\W$ (see \cite[Corollary 1.10]{MR3995092}); in particular, this shows that $\lambda_1^L$ can be zero in some domains and therefore the second alternative in Theorem~\ref{loglap:cor} can happen.

Theorem \ref{loglap:cor} follows from the more general Theorem \ref{fredholm} below. In particular, an analogous result holds for $L_\Delta+T$, where $T$ is a compact perturbation; for example, since the inclusion from $\mathcal X^\alpha(\Omega)$ into $\mathcal Y(\Omega)$ is compact (see Lemma \ref{cpt:local}), $T$ can be a multiple of the identity.

Theorem \ref{loglap:cor} does not aim at finding the optimal values for $\alpha$, it merely guarantees the existence of classical solutions in the natural setting of log-Hölder spaces.  However, the Green's functions estimates given in \cite{kim2014green} suggest that the optimal $\alpha$ for global regularity (in $\R^N$) is $\alpha=\frac{1}{2}$ whenever $f$ is smooth (see also \cite[Theorem 1.11]{MR3995092}). On the other hand, we conjecture that a bound for
\[
\sup_{\substack{x,y\in \W\\x\neq y}}\ell^{1+\alpha}(d(x,y))\frac{|u(x)-u(y)|}{\ell^{1+\alpha}(|x-y|)},
\]
can be given in terms of
\begin{align*}
\|f\|_{L^\infty(\W)}+\sup_{\substack{x,y\in \W\\x\neq y}}\ell^{1+\alpha}(d(x,y))\frac{|f(x)-f(y)|}{\ell^\alpha(|x-y|)}.
\end{align*}
One of the main obstacles to show this conjecture is the lack of good scaling properties in logarithmic operators (see Section \ref{sec:1}).

Our general results also cover other logarithmic problems of importance, such as the \emph{logarithmic Schrödinger equation}. The logarithmic Schrödinger operator (see \cite{bass94,beghin14,feulefack2021logarithmic,MR3626549,kozubowski99,vsikic06} and the references therein) is given by
\begin{align}\label{Schlog}
(I-\D)^{\log} u(x) := \int_{\R^N} \frac{u(x)-u(y)}{|y-x|^N}\w(|y-x|)dy,
\end{align}
where $\w(r) = c_Nr^{N/2}K_{N/2}(r)$ and $K_{N/2}$ is a modified Bessel function of second kind and index $N/2$.  In particular, $\w(r) = c_N+ b_Nr + o(r)$ as $r\to 0^+$.  We show the following.

\begin{theorem}\label{Sch:log}
Let $\W\ss\R^N$ be an open bounded Lipschitz set satisfying a uniform exterior ball condition. There is $\alpha=\alpha(N,\W)\in(0,1)$ such that, for every $f\in \mathcal Y(\Omega)$, there exists a unique classical solution $u\in \mathcal X^\alpha(\W)$ of
\begin{align*}
(I-\D)^{log} u = f \text{ in } \W,\qquad
u=0 \text{ in } \R^N\sm\W.
\end{align*}
Moreover, $\|u\|_{\mathcal X^\alpha(\W)}\leq C\|f\|_{\mathcal Y(\W)}$ for some constant $C=C(N,\W)>0$.
\end{theorem}

Similarly as before, one can also consider the operator $(I-\D)^{log}$ plus a compact perturbation and a Fredholm-Alternative-type result follows from Theorem \ref{fredholm}.

\subsection{General framework}

Next, we present our general setting to handle a wide variety of logarithmic-type integrodifferential equations. Let $\W\ss\R^N$ be a bounded open set and let $B_r$ denote the open ball of radius $r>0$ centered at 0.  For $K \in L^\8(\W\times B_1)$ and for a suitable $u$, let 
\begin{align*}
L_Ku(x) := \int_{B_1(x)} \frac{u(x)-u(y)}{|y-x|^N}K(x,y-x)\, dy.    
\end{align*}

Observe that this operator $L_K$ is nonlocal but it has a \emph{finite range of interaction}, since the integral is over the bounded set $B_1(x)$.  This is one way of dealing with the fact that the kernel $|y|^{-N}$ is not integrable at $\infty$.  Note, for instance, that the integral 
\begin{align*}
\int_{\R^N} \frac{u(x)-u(y)}{|y-x|^N}\, dy
\end{align*}
may not be finite for any $x\in \R^N$ even if $u \in C^\8_c(\R^N)$.   Another (equivalent) approach to solve this integrability issue (used for example in \cite{MR3626549}) would be to impose some decay on the mapping $y\mapsto K(x,y)$ as $|y|\to \infty$, but we believe that the use of a finite range of interaction is simpler and more natural for our applications. 

We aim at finding weak assumptions on $K$ to guarantee the existence of classical solutions of the associated Dirichlet problem in the natural setting of log-Hölder spaces. We say that $L_K$ is \emph{uniformly elliptic with respect to $0<\l\leq \L$} if
\[
\l\leq K(x,y)\leq \L \text{ for all } (x,y)\in\W\times B_1.
\]

Uniform ellipticity endows the operator $L_K$ with a positivity preserving property (see Lemma~\ref{lem:mp}) and it is the key ingredient to obtain a first (lower-order) regularity estimate (see Theorem \ref{thm1}) which, under Dirichlet boundary conditions, can be extended to a global regularity estimate (see Theorem \ref{thm:bdry}). Here, by Dirichlet boundary conditions we mean that
\begin{align*}
u=0 \text{ in $\R^N\backslash \Omega$}.
\end{align*}
Of course, due to the finite range of interaction, it would suffice to impose this condition on $B_1(\Omega)\backslash \Omega$ (with $B_1(\Omega):=\{x\in \R^N\::\:\dist(x,\Omega)<1\}$), but to simplify notation we always consider the trivial extension of $u$ to $\R^N$.

Moreover, to obtain the existence of classical solutions, namely, functions that allow the pointwise evaluation of the operator, we require higher-order regularity estimates (see Theorem \ref{thm2}). For this, we need to impose further assumptions on $K$.  We say that $L_K$ is \emph{translation invariant} if
\[
K(x,y) = K(y) \text{ for all } (x,y)\in\Omega\times B_1
\]

The following is our main existence result.  Let $DK$ denote the gradient of $K$.
\begin{theorem}[Fredholm Alternative]\label{fredholm}
Let $\W\ss\R^N$ be a bounded Lipschitz open set with a uniform exterior ball condition, $L_K$ be a translation invariant uniformly elliptic operator  with respect to $0<\l\leq \L$, and assume that $K$ is differentiable with
\begin{align}\label{1reg}
|D K(y)|\leq \L|y|^{-1} \text{ for every $y\in\R^N\sm\{0\}$}.
\end{align}
Then, there is $\alpha=\alpha(N,\l,\L,\Omega)>0$ such that, for every $T: \mathcal X^\alpha(\Omega) \to \mathcal Y(\Omega)$ compact linear operator, exactly \underline{one} of the following alternatives holds:
\begin{enumerate}
    \item For every $f \in \mathcal Y(\Omega)$ there exists a unique classical solution $u\in\mathcal X^\alpha(\Omega)$ of
    \begin{align}\label{fa}
(L_K+T)u=f \text{ in } \W,\qquad
u=0 \text{ on } \R^N\sm\W;
\end{align}
moreover, $\|u\|_{\mathcal X^\alpha(\Omega)} \leq C\|f\|_{\mathcal Y(\Omega)}$ for some $C=C(N,\l,\L,\W)>0$.
    \item There is a non-trivial classical solution $u\in \mathcal X^\alpha(\Omega)$ of $(L_K+T)u=0$.
\end{enumerate}
In particular, the first alternative always holds for $T\equiv 0$.
\end{theorem}

Theorems \ref{loglap:cor} and \ref{Sch:log} follow from Theorem \ref{fredholm} by choosing a suitable compact operator $T$ (see Lemma~\ref{T:comp}). Therefore, the Fredholm Alternative is a powerful tool to transfer the regularity results obtained for the finite-range-interaction operator $L_K$ to other more general integrodifferential operators of logarithmic type with infinite range of interaction.  

We now briefly discuss the main ideas behind Theorem \ref{fredholm} and the organization of the paper.

In Section \ref{sec:1} we begin with a  brief discussion on the lack of good scaling properties of these operators, which is one of the main obstacles to overcome.  In Section \ref{sec:2} we establish that $L_K$ is a bounded operator between appropriate log-Hölder spaces (Theorem \ref{prop3}). This result is optimal in terms of the exponents involved.  In Section \ref{sec4} we show the lower-order initial regularity estimates (Theorem \ref{thm1}). These estimates were shown first by Kassmann and Mimica in \cite{MR3626549} using a diminish of oscillation strategy inspired by probabilistic techniques \cite{kmarx}. We revisit this technique with a slight variation (we do not require an even assumption on the kernel, for instance).

Starting in Section \ref{hoe} we establish the main regularity estimates of this work. These extend the Krylov-Safonov theory to higher order interior estimates and boundary estimates respectively. In Section \ref{brwzbd} we use the method of barriers to establish boundary regularity estimates for open sets with an exterior ball condition.  Here, we show that the barriers obtained in \cite{MR3995092} can also be used for the operator $L_K$.

Finally, in Section \ref{dbvp}, we use our regularity estimates together with the very recent existence result by \cite{feulefack2021nonlocal} to settle the existence of classical solutions of the equation $L_Ku = f$ in a Lipschitz open bounded set with exterior ball condition and zero complementary data in log-Hölder spaces.

\subsection{Acknowledgments}

H. Chang-Lara  is supported by CONACyT-MEXICO grant A1-S-48577. A. Saldaña is supported by UNAM-DGAPA-PAPIIT grant IA101721
and by CONACYT grant A1-S-10457, Mexico.

\section{On the lack of a scale invariance property}\label{sec:1}

We begin with a discussion on the main obstacle to overcome when dealing with zero order integrodifferential operators, which is the lack of a scale invariance property. To explain this point let us recall how this is commonly used in the case of operators of order $s=2$ to obtain a Hölder estimate, a strategy known as the \textit{growth lemma} or \textit{diminish of oscillation lemma}, which can be traced back to the work of Landis \cite{MR1487894}.

Let $0<\l\leq \L$ and $(a_{ij}(x))$ be a function from the unit ball $B_1 \ss\R^N$ to the set of $N$ by $N$ symmetric matrices such that, for every $x \in B_1$ and $\xi\in \R^N$,
\[
\l|\xi|^2 \leq \sum_{i,j=1}^n a_{ij}(x)\xi_i\xi_j \leq \L|\xi|^2.
\]
The goal is to show that, for every $u \in C^2(B_1)$ such that $\sum_{i,j=1}^N a_{ij}\p_{ij}u = 0$, there is an $\a$-Hölder modulus of continuity at the origin of the form
\[
\sup_{r\in(0,1)}r^{-\a}\osc_{B_r} u \leq C\osc_{B_1} u \qquad (\osc_E u := \sup_E u - \inf_E u),
\]
where $\a\in(0,1)$ and $C>0$ are universal constants, meaning that they only depend on the dimension $N$ and on the ellipticity constants $\l$ and $\L$.

In order to establish these estimates for every scale $r\in(0,1)$, it actually suffices to show just one case. For instance, it can be proved that for $r=1/2$ it holds that
\[
\sum_{i,j=1}^N a_{ij}\p_{ij}u = 0 \text{ in } B_1 \quad\implies\quad \osc_{B_{1/2}}u \leq (1-\theta)\osc_{B_1} u,
\]
for some $\theta\in(0,1)$ depending only on $N$, $\l$, and $\L$, but not on the particular choice of $(a_{ij})$. Once this statement is known to be true, we can verify that the rescaled function $v(x) := u(x/2)$ satisfies a similar equation and hence
\[
\osc_{B_{1/4}}u = \osc_{B_{1/2}}v \leq (1-\theta)\osc_{B_1} v = (1-\theta)\osc_{B_{1/2}} u \leq (1-\theta)^2\osc_{B_1} u.
\]
Inductively, for $k\in\N$,
\[
\osc_{B_{2^{-k}}}u \leq (1-\theta)^k\osc_{B_1} u,
\]
from where we get that, for $\a = |\ln(1-\theta)/\ln 2|$ and $C=(1-\theta)^{-1}$,
\[
\osc_{B_{r}}u \leq Cr^\a\osc_{B_1} u.
\]

A crucial step in this approach has been that as the rescaling takes place, the new function satisfies again an elliptic equation with a uniform assumption on the coefficients (the quadratic forms get controlled from above and below in terms of the \underline{same} constants $\l$ and $\L$).

Let us see now how does the equation scales for $L_K$ with the simplest kernel $K = 1$. For $u_r(x) := u(rx)$ with $r\in(0,1)$, we obtain the following identity for any $x\in B_1$
\begin{align}\label{eq:sca}
L_Ku(rx) = \int_{B_1(rx)} \frac{u(rx)-u(y)}{|y-rx|^N}dy = \int_{B_{1/r}(x)} \frac{u_r(x)-u_r(z)}{|z-x|^N}dz =: L_{K,r}u_r(x).
\end{align}
This means that the information that we may originally know about $L_Ku$ in $B_{r}$ transfers to $L_{K,r}u_r$ in $B_1$, but notice that the kernel has spread its support to $B_{1/r}$, which, as mentioned in the Introduction, yields integrability issues as $r\to 0$.

This lack of scaling certainly prevents the possibility of having regularity estimates in the standard Hölder spaces. Going back to \eqref{eq:sca}, we could try to fix this issue by splitting the integral as
\[
L_{K,r}u_r(x) = L_Ku_r(x) + \int_{B_1(rx)\sm B_r(rx)} \frac{u(rx)-u(y)}{|y-rx|^N}dy.
\]
This means that the previous strategy now needs to incorporate right-hand side terms which will be given to the rescaled equation after each iteration.

Assuming that $L_Ku=f$ in $B_1$ and $u$ admits a modulus of continuity $\w$, we get that the right-hand side for the rescaled equation could be written as 
\[
L_{K}u_r = f_r \text{ in } B_1,
\]
where
\[
f_r(x) := f(rx) - \int_{B_1(rx)\sm B_r(rx)} \frac{u(rx)-u(y)}{|y-rx|^N}dy
\]
satisfies that
\[
\|f_r - f(r\cdot)\|_{L^\8(B_1)} \leq C\int_r^1\frac{\w(\r)}{\r}d\r.
\]

Any modulus of continuity $\w$ for which the previous integral remains bounded as $r\to 0^+$ are said to be of Dini type. Functions that admit Dini modulus of continuity are called Dini continuous and we use $C^{Dini}(\W)$ to denote the set of these functions over some open set $\W$. This is also the least regularity we need in order to evaluate our operators in the classical sense, as can be verified from the definition of $L_K$. For example, the Hölder modulus $\w(\r) = \r^\a$ for any $\a\in(0,1]$, is of Dini type. However a tighter and more natural choice is to take
\[
\w(\r) = \ell^{1+\a}(\r) = |\ln(\min(\r_0,\r))|^{-(1+\a)} \qquad (\r_0=0.1, \a>0),
\]
given that the measure in the integral $d\r/\r = d(\ln \r)$ is integrable in the interval $[0,\r_0] = [0,0.1]$ against the function $|\ln\r|^{-(1+\a)}$ for $\a>0$. Indeed we use in several computations that, for $r \in (0,\r_0)$,
\[
\int_0^r \frac{\ell^{1+\a}(\r)}{\r}d\r = -\int_0^r\frac{d|\ln \r|}{|\ln \r|^{1+\a}} = \left.\frac{|\ln \r|^{-\a}}{\a}\right|_0^r = \frac{\ell^\a(r)}{\a}.
\]
We hope that this brief motivation may help to reveal the central role of the log-Hölder moduli of continuity in our work.

\section{Preliminaries}\label{sec:2}

In the following $\Omega\subset \R^N$ is an open subset of the Euclidean space of dimension $N\geq 1$. We denote the ball with center at $x$ and radius $r$ as $B_r(x)$, and $B_r := B_r(0)$. Moreover, $\omega_N:=|\partial B_1|=\frac{N\pi^\frac{N}{2}}{\Gamma(\frac{N}{2}+1)}$ denotes the measure of the surface of the unit ball and 
 \[
 B_1(\W) := \bigcup_{x\in \W} B_1(x).
 \]

\subsection{Log-Hölder moduli of continuity}\label{sec2.1}

\begin{definition}
As defined in the introduction, we fix
\[
\ell(\r) = |\ln(\min(\r_0,\r))|^{-1}, \qquad\r_0:=0.1. 
\]
where the truncation guarantees that $\ell$ is a modulus of continuity. That is to say that $\ell$ is a non-decreasing, concave function over the interval $(0,\8)$, taking values in $(0,1]$, and with the limit $\ell(0) := \lim_{r\to0^+}\ell(r)=0$.
\end{definition}

We use that $\ell$ satisfies the following property. Its proof, among other properties for this modulus of continuity are included in the appendix.

\begin{lemma}[Semi-homogeneity]\label{prop1}
There is $c>0$ such that 
\[
\ell(\l)\leq c\frac{\ell(\l r)}{\ell(r)}\,\qquad \text{ for all }r,\l >0.
\]
\end{lemma}

\begin{remark}
A reverse type inequality is not possible as it can be seen by taking $\l=r\to0^+$.
(meanwhile $\ell(\l r)/\ell(r) = 1/2$ we get that $\ell(\l)$ is arbitrarily small). 
Instead, we have the useful inequality
\[
c^{-1}\,\ell(\l) \leq \inf_{r>0}\frac{\ell(\l r)}{\ell(r)}\leq \sup_{r>0}\frac{\ell(\l r)}{\ell(r)} \leq \frac{c}{\ell(1/\l)}.
\]
\end{remark}

For $\a\in(0,1]$, $\ell^\a$ is also a modulus of continuity with properties similar to those we have already mentioned for $\ell$ (monotonicity, concavity, semi-homogeneity, etc.).

\begin{definition}
Let $\W\ss\R^n$ open,
\begin{align*}
d(x):=\dist(x,\partial \Omega),\quad \text{ and }\quad d(x,y):=\min(d(x),d(y)).    
\end{align*}
 Given $\a,\beta\geq 0$ we construct the semi-norms
\begin{align*}
&[u]_{\mathcal L^\a(\W)} := \sup_{\substack{x,y\in \W\\x\neq y}}\frac{|u(x)-u(y)|}{\ell^\a(|x-y|)}, \\
&[u]^{(\b)}_{\mathcal L^\a(\W)} := \sup_{\substack{x,y \in \W\\x\neq y}} \ell^{\a+\b}(d(x,y))\frac{|u(x)-u(y)|}{\ell^\a(|x-y|)},
\end{align*}
and the norms
\begin{align*}
&\|u\|_{\mathcal L^\a(\W)} := \|u\|_{L^\8(\W)} + [u]_{\mathcal L^{\a}(\W)},\\
&\|u\|^{(\b)}_{\mathcal L^\a(\W)} := \sup_{x\in\W} \ell^{\b}(d(x))|u(x)| + [u]^{(\b)}_{\mathcal L^\a(\W)}.
\end{align*}
We also define the Banach space (Lemma \ref{banach})
\begin{align*}
\mathcal L^\a(\W):=\{u:\W\to\R\::\: \|u\|_{\cL^\alpha(\Omega)}<\infty\}
\end{align*}
 and its local version
\[
\mathcal L^\a_{loc}(\W) := \{u \in C(\W) \ | \ u \in \mathcal L^\a(B_r(x_0)) \ \text{ for all } \, \overline{B_r(x_0)} \ss\W\}.
\]
\end{definition}

Notice that, different to the classical Hölder spaces, our moduli of continuity encompass all the ranges of regularity $\a\geq0$, without separating the cases according to the integer part of $\a$. However, the threshold $\a=1$ still plays an important role in our results. 

For $\a=0$, the semi-norm $[\cdot]_{\mathcal L^0(\W)}$ measures the oscillation
\[
[u]_{\mathcal L^0(\W)} = \sup_{x,y\in\W} |u(x)-u(y)| = \sup_\W u-\inf_\W u =: \osc_\W u.
\]
The Banach space $\mathcal L^0(\W)$ is equivalent to $L^\8(\Omega)$.

A function $u$ belongs to the local space $\mathcal L^\a_{loc}(\W)$ whenever $\|u\|^{(\b)}_{\mathcal L^\a(\W)}<\8$ for some $\b\geq 0$. It may well be also the case that $u\in\mathcal L^\a_{loc}(\W)$, although $\|u\|^{(\b)}_{\mathcal L^\a(\W)}$ diverges for every $\b\geq0$. However, in this work, every $u\in \mathcal L^\a_{loc}(\W)$ is complemented with a bound on $\|u\|^{(\b)}_{\mathcal L^\a(\W)}$ for some $\b\geq 0$.

The exponent $\b=0$ gives the analogous to the non-dimensional norms and semi-norms as in the book by Gilbarg and Trudinger \cite[p. 61]{Gilbarg-Trudinger2001}. In the second order case, different exponents $\b>0$ can be used to compensate the homogeneities of operators with scale invariant properties. This allows to express interior estimates in terms of these weighted norms, c.f. \cite[p. 66]{Gilbarg-Trudinger2001}. In our case it plays a similar role even though the operators do not enjoy a scale invariance property, see for instance the statement in Theorem \ref{prop3} where different $\b$'s appear in the estimate.

The following lemma allows us to get estimates for $\|\cdot\|^{(\b)}_{\mathcal L^\a(\W)}$ once we get interior estimates over balls. A proof can be found in the appendix.

\begin{lemma}\label{prop2}
Let $\W\ss\R^N$ be an open set with bounded inner radius $R:=\sup_{x\in\W}\dist(x,\p\W)<\8$, $\a,\b\geq 0$, $r_0>0$, and $\l>1$. Given $u:\W\to \R$ the following estimates hold (meaning that either all the quantities are finite and the inequalities hold or each of the three sides is infinite),
\begin{align*}
&c\|u\|^{(\b)}_{\mathcal L^\a(\W)} \leq \sup_{\substack{B_{\l r}(x)\ss\W\\r\in(0,r_0)}} \1\ell^{\b}(r)\|u\|_{L^\8(B_{r}(x))} + \ell^{\a+\b}(r)[u]_{\mathcal L^{\a}(B_{r}(x))}\2 \leq C\|u\|^{(\b)}_{\mathcal L^\a(\W)},
\end{align*}
where $c,C>0$ depend only on $N$, $R$, $\a$, $\b$ , $r_0$, and $\l$.
\end{lemma}

The following lemma is used to glue together interior and boundary estimates. Once again, its proof can be found in the appendix. Recall that $d(x):=\dist(x,\partial \Omega)$.

\begin{lemma}\label{prop4}
Given $\W\ss\R^N$ a bounded open set, $\a,\b\geq 0$, and $\w_0:[0,\8)\to[0,\8)$ a modulus of continuity, there exists a modulus of continuity $\w:[0,\8)\to[0,\8)$ depending on $\a$, $\b$, $\w_0$ and $\W$ such that the following holds: Given $u \in \mathcal L^\a_{loc}(\W)\cap C(\overline\W)$ such that the right-hand side below is finite, then the left-hand side is also bounded and
\begin{align*}
\sup_{\substack{x,y\in\overline\W\\x\neq y}}\frac{|u(x)-u(y)|}{\w(|x-y|)} \leq [u]^{(\b)}_{\mathcal L^\a(\W)} + \sup_{\substack{x\in\p\W\\y\in\overline\W\sm\{x\}}}\frac{|u(x)-u(y)|}{\w_0(|x-y|)}.
\end{align*}
In particular, if $\w_0 = \ell^{\a'}$ then $\w = C\ell^{\gamma}$ for $C$ and $\gamma$ depending on $\a$, $\a'$, $\b$, $\w_0$ and $\W$.
\end{lemma}

\subsection{Ellipticity}\label{subsec:dini}

Recall that $C^{Dini}(\W)$ is the collection of functions which admit a modulus of continuity $\w$ such that
\[
\int_0^1\frac{\w(\r)}{\r}d\r <\8.
\]
Its local version $C^{Dini}_{loc}(\W)$ denotes the set of functions $u:\W\to\R$ such that $u\in C^{Dini}(B_r(x_0))$ for every $\overline{B_r(x_0)}\ss\W$. In particular, $x\mapsto L_K u(x)$ is well defined in $\W$ if $u\in C^{Dini}_{loc}(\W)$ (see, for example, the proof of \cite[Proposition 1.3]{MR3995092}).

\begin{definition}\label{def}
A kernel on $\W\ss\R^N$ is a bounded measurable function $K \in L^\8(\W\times B_1)$. It is used to construct a linear operator $L_K$ that can be evaluated at $u \in L^1(B_1(\W))\cap C^{Dini}_{loc}(\W)$ over $x\in\W$ in the following way
\[
L_Ku(x) := \int_{B_1(x)} \frac{u(x)-u(y)}{|y-x|^N}K(x,y-x)dy.
\]
\end{definition}

As mentioned in the introduction, we may impose the following additional properties.
\begin{itemize}
    \item[\namedlabel{itm:ellip}{\textbf{(E)}}] We say that $L_K$ is \emph{elliptic} if $K$ is non-negative.
    \item[\namedlabel{itm:uellip}{\textbf{(UE)}}] We say that $L_K$ is \emph{uniformly elliptic with respect to $0<\l\leq \L$} if $\l\leq K(x,y)\leq \L$ for all $(x,y)\in\W\times B_1$.
    \item[\namedlabel{itm:ti}{\textbf{(TI)}}] We say that $L_K$ is \emph{translation invariant} if $K(x,y) = K(y)$, namely, if $K$ only depends on its second entry $y\in B_1$.
\end{itemize}

Ellipticity is the fundamental feature which allows us to implement comparison techniques in our analysis. By this we mean that the operator enjoys a generalization of the second derivative test, which in the classical second order case can be stated in the following way: Whenever the graphs of two smooth ordered functions are in contact at a given point, their Hessians are ordered as well at the contact point. Hence any monotone function of the Hessian (e.g. the Laplacian) is also ordered at the contact point.

\begin{definition} Let $\W_v,\W_u\subset \R^N$ be open sets, $\Omega\ss  \W_u\cap\W_v$ open, and let $x_0\in\Omega$. We say that a function $v:\W_v\to \R$ touches $u:\W_u\to\R$ from above over $\Omega$ at the so called contact point $x_0$, if $v\geq u$ in $\Omega$ and $u(x_0)=v(x_0)$.
\end{definition}

This concept could also be expressed by saying that $u$ touches $v$ from below at $x_0$ over $\Omega$.

The following lemma reflects the ellipticity of our operators in the same spirit of the second derivative test, but adapted to our integrodifferential operators. Notice that the convention of ellipticity that we follow falls in the class where \textit{minus} the Laplacian is considered to be elliptic.

\begin{lemma}\label{prop:elip}
Let $r>0$, $x_0\in\R^N$, $K\in L^\infty(B_r(x_0)\times B_1)$, and let $L_K$ satisfy the ellipticity condition \ref{itm:ellip}.  Let $u,v\in L^1(B_1(x_0))\cap C^{Dini}_{loc}(B_r(x_0))$. If $v$ touches $u$ from above over $B_1(x_0)$ at $x_0$, then $L_Ku(x_0)\geq L_Kv(x_0)$.
\end{lemma}
\begin{proof}
Let $w:=u-v\leq 0$.  Then $w$ has a global maximum in $B_1(x_0)$ at $x_0$ and
\[
L_Ku(x_0) - L_Kv(x_0) = L_Kw(x_0) = \int_{B_1(x_0)}\frac{w(x_0)-w(y)}{|y-x_0|^N}K(x_0,y-x_0)dy \geq 0.
\]
\end{proof}

\subsection{Boundedness}\label{sec:bdd}

The main goal of this section is to show that under suitable hypotheses, $L_K$ maps $L^\8(B_1(\W))\cap \mathcal L^{1+\a}_{loc}(\W)$ to $\mathcal L^{\a}_{loc}(\W)$. Our proof requires a technical assumption on the oscillation of the kernel (recall that even in the second order case, the map $u\in C^{2,\a}(\W) \mapsto \sum_{i,j=1}^n a_{ij}(x)\p_{ij}u(x) \in C^\a(\W)$ requires that the coefficients $(a_{ij})$ are at least $\a$-Hölder continuous). 

\begin{definition}\label{def:reg}
Let $\W\subset \R^N$ be a bounded open set and let $K\in L^\infty(\W\times B_1)$ be a kernel extended by zero in $\W\times (\R^N\sm B_1)$. Given $\a >0$, we say that the operator $L_K$ is $\a$-regular with constant $\L\geq 0$ if, for every $z \pm \frac{w}{2}\in\W$,
\begin{align*}
 \int_{\R^N\sm B_{2|w|}} \left|\frac{K(z+\frac{w}{2},\xi+\frac{w}{2})}{|\xi+\frac{w}{2}|^N}-\frac{K(z-\frac{w}{2},\xi-\frac{w}{2})}{|\xi-\frac{w}{2}|^N}\right|\ell(|\xi|)d\xi \leq \L\ell^\a(|w|).    
\end{align*}
\end{definition}

In order to skip some technicalities on a first reading of the following theorem, one may omit the $\a$-regularity assumption on the kernel and just consider the case $K(x,y) = \chi_{B_1}(y)$.

\begin{theorem}\label{prop3}
Let $\a>0$, $\W\ss\R^N$ be an open bounded set, $K\in L^\infty(\Omega\times B_1)$, and let $L_K$ be $\a$-regular (with constant $\L_0$) and satisfy the ellipticity condition \ref{itm:ellip}. For any $u \in L^\8(B_1(\W))\cap \mathcal L^{1+\a}_{loc}(\W)$ such that $[u]_{\mathcal L^{1+\a}(\W)}^{(0)} <\8$, it holds that $L_K u \in \mathcal L^{\a}_{loc}(\W)$ and 
\[
\|L_Ku\|^{(1)}_{\mathcal L^\a(\W)} \leq C([u]_{\cL^0(B_1(\W))} + [u]_{\mathcal L^{1+\a}(\W)}^{(0)})
\]
for some $C>0$ depending only on $\a$, $N$, $\L_0$, $\|K\|_{L^\infty(\Omega\times B_1)}$, and $\diam(\W)$.
\end{theorem}

\begin{proof}
By Lemma \ref{prop2}, it suffices to show that there exists some radius $r_0\in (0,\r_0)$\footnote{Recall that $\r_0=0.1$ and $\ell(r) = |\ln r|^{-1}$ for $r\in(0,\r_0)$.} sufficiently small such that, for any $r \in (0,r_0)$ and $B_{6r}(x_0) \ss \W$,
\[
\ell(r)\|L_Ku\|_{L^\8(B_r(x_0))} + \ell^{1+\a}(r)[L_Ku]_{\mathcal L^{\a}(B_r(x_0))}\leq C( [u]_{\cL^0(B_1(\W))} + \ell^{1+\a}(r)[u]_{\mathcal L^{1+\a}(B_{5r}(x_0))}).
\]

Let $\Lambda:=\max\{\|K\|_{L^\infty(\Omega\times B_1)},\L_0\}>0$.

\textbf{$L^\8$ term:} For $x\in B_r(x_0)$,
\begin{align*}
|L_Ku(x)| &\leq \int_{B_1(x)} \frac{|u(x)-u(\xi)|}{|\xi-x|^N}K(x,\xi-x)d\xi\\
&\leq \omega_N\L [u]_{\mathcal L^{1+\a}(B_{5r}(x_0))}\int_0^r \frac{\ell^{\a+1}(\r)}{\r}d\r + \omega_N\L[u]_{\cL^0(B_1(\W))}\int_r^1\frac{d\r}{\r}\\
&= \frac{\omega_N\L}{\alpha}\ell^{\a}(r)[u]_{\mathcal L^{1+\a}(B_{5r}(x_0))}+\omega_N\L\ell^{-1}(r)[u]_{\cL^0(B_1(\W))}.
\end{align*}
Taking the supremum on $x\in B_r(x_0)$ and multiplying by $\ell(r)$,
\[
\ell(r)\|L_Ku\|_{L^\8(B_r(x_0))} \leq C([u]_{\cL^0(B_1(\W))} + \ell^{1+\a}(r)[u]_{\mathcal L^{1+\a}(B_{5r}(x_0))}).
\]

\textbf{$\mathcal L^\a$ term:} Let $x,y\in B_r(x_0)$, $\delta:=|x-y|<2r$, and $z:=\frac{x+y}{2}\in B_r(x_0)$.

For $r_0\leq 1/5$ we guarantee that the following inclusions hold true,
\[
B_{2\delta}(z) \ss B_{1-\delta/2}(z) \ss B_1(x) \cup B_1(y) \ss B_{1+\delta/2}(z),
\]
this allows us to split $Lu(x)-Lu(y)$ in the following three terms
\begin{align*}
Lu(x)-Lu(y) &= J_1+J_2+J_3,
\end{align*}
where 
\begin{align*}
J_1&:=\int_{B_{2\delta}(z)}I(x,y,\xi) d\xi,\\
J_2&:=\int_{B_{1-\delta/2}(z)\sm B_{2\delta}(z)}I(x,y,\xi) d\xi,\\
J_3&:=\int_{B_{1+\delta/2}(z)\sm B_{1-\delta/2}(z)}I(x,y,\xi) d\xi,\\
I(x,y,\xi) &:=\frac{u(x)-u(\xi)}{|\xi-x|^N}K(x,\xi-x)-\frac{u(y)-u(\xi)}{|\xi-y|^N}K(y,\xi-y).
\end{align*}
Recall that $K(z,\cdot)$ is extended by zero outside $B_1$. Our goal is to bound the $J_i$'s by 
\begin{align*}
\ell^\a(\delta)(\ell^{-1-\a}(r)[u]_{\cL^0(B_1(\W))}+[u]_{\cL^{1+\a}(B_{5r}(x_0))}).    
\end{align*}

For $J_3$, the integrand is bounded by a multiple of $[u]_{\cL^0(B_1(\W))}$ and the domain of integration has measure proportional to $\delta$. Hence,
\[
J_3 \leq C\delta[u]_{\cL^0(B_1(\W))} \leq C\ell^\a(\delta)[u]_{\cL^0(B_1(\W))}.
\]
 
In $J_1$ we use the $\cL^{1+\a}$ control over $B_{2\delta}(z) \ss B_{5r}(x_0)$:
\begin{align*}
J_1 \leq 2 \L\omega_N [u]_{\mathcal L^{1+\a}(B_{5r}(x_0))}\int_0^{2\delta}\frac{\ell^{1+\a}(\r)}{\r}d\r \leq 4\L\frac{\omega_N}{\alpha}\ell^\a(\delta)[u]_{\mathcal L^{1+\a}(B_{5r}(x_0))}.
\end{align*}

In $J_2$ we use the discrete product rule for the integrand:
\[
I(x,y,\xi) = \underbrace{\frac{u(x)-u(y)}{|\xi-x|^N}K(x,\xi-x)}_{=:I_1(x,y,\xi)} + \underbrace{(u(y)-u(\xi))\1\frac{K(x,\xi-x)}{|\xi-x|^N}-\frac{K(y,\xi-y)}{|\xi-y|^N}\2}_{=:I_2(x,y,\xi)}.
\]

For $I_1$ we are allowed to replace the denominator $|\xi-x|\sim|\xi-z|$, because $x \in B_\delta(z)$ and $\xi \in \R^N\sm B_{2\delta}(z)$. Hence,
\[
\int_{B_{1-\delta/2}(z)\sm B_{2\delta}(z)} |I_1(x,y,\xi)|d\xi \leq C\frac{\ell^{1+\a}(\delta)}{\ell(2\delta)}[u]_{\mathcal L^{1+\a}(B_{5r}(x_0))} \leq C\ell^{\a}(\delta)[u]_{\mathcal L^{1+\a}(B_{5r}(x_0))}. 
\]

Finally, to estimate the integral of $I_2$ we use the $\a$-regularity hypothesis (note that the case $K=1$ is simpler and follows using the fundamental theorem of calculus, see Lemma \ref{lem:alphareg})
\begin{align*}
&\int_{B_{1-\delta/2}(z)\sm B_{2\delta}(z)}\left|\frac{K(x,\xi-x)}{|\xi-x|^N}-\frac{K(y,\xi-y)}{|\xi-y|^N}\right|\ell(|\xi-z|)d\xi\\
&=\int_{B_{1-\delta/2}\sm B_{2\delta}}\left|\frac{K(x,\xi+\frac{y-x}{2})}{|\xi+\frac{y-x}{2}|^N}-\frac{K(y,\xi+\frac{x-y}{2})}{|\xi+\frac{x-y}{2}|^N}\right|\ell(|\xi|)d\xi
\leq C\ell^\a(\delta).
\end{align*}

This allows us to finally estimate the integral of $I_2$ by splitting the domain of integration $B_{1-\delta/2}(z)\sm B_{2\delta}(z)$ into $B_{4r}(z)\sm B_{2\delta}(z)$ and $B_{1-\delta/2}(z)\sm B_{4r}(z)$. Over the small annulus we bound $|u(y)-u(\xi)|$ using the control in $\cL^{1+\a}(B_{5r(x_0)})$:
\begin{align*}
 &\int_{B_{4r}(z)\sm B_{2\delta}(z)} |I_2(x,y,\xi)|d\xi\\
&\leq C[u]_{\cL^{1+\a}(B_{5r}(x_0))}\int_{B_{1-\delta/2}(z)\sm B_{2\delta}(z)} \left|\frac{K(x,\xi-x)}{|\xi-x|^N}-\frac{K(y,\xi-y)}{|\xi-y|^N}\right|\ell^{1+\a}(|\xi-z|)d\xi,\\
&\leq C[u]_{\cL^{1+\a}(B_{5r}(x_0))}\ell^\a(\delta)
\end{align*}
and over the larger annulus we use the bound for the oscillation:
\begin{align*}
&\int_{B_{1-\d/2}(z)\sm B_{4r}(z)} |I_2(x,y,\xi)|d\xi\\ &\leq C[u]_{\cL^0(B_1(\W))}\ell^{-1}(r)\int_{B_{1-\delta/2}(z)\sm B_{4r}(z)} \left|\frac{K(x,\xi-x)}{|\xi-x|^N}-\frac{K(y,\xi-y)}{|\xi-y|^N}\right|\ell(|\xi-z|)d\xi,\\
&\leq C[u]_{\cL^0(B_1(\W))}\ell^{-1}(r)\ell^\a(\delta).
\end{align*}
In any case,
\begin{align*}
J_2\leq \int_{B_{1-\delta/2}(z)\sm B_{2\delta}(z)} |I_2(x,y,\xi)|d\xi \leq C\ell^\a(\delta)(\ell^{-1}(r)[u]_{\cL^0(B_1(\W))}+[u]_{\cL^{1+\a}(B_{5r}(x_0))}).
\end{align*}
This ends the proof.
\end{proof}

Here is a practical way to verify if $L_K$ is 1-regular (and thus $\a$-regular for any $\a\in(0,1]$).

\begin{lemma}\label{lem:alphareg}
Let $\W\ss\R^N$ be a convex set. If $K \in L^\infty(\W\times B_1)\cap C^1(\W\times (B_1\sm\{0\}))$ is a kernel such that 
\begin{align}\label{J:a}
|D_x K(x,y)|+|D_yK(x,y)| \leq \L |y|^{-1}\qquad \text{for every $(x,y) \in \W\times (B_1\sm\{0\})$},
\end{align}
then $L_K$ is $1$-regular.
\end{lemma}

\begin{remark}
Notice that, for $K$ translation invariant, the assumption \eqref{J:a} plays a role in the existence result in \cite[Theorem 1.6]{feulefack2021nonlocal} and, as a consequence, also in our main result Theorem \ref{fredholm}. We conjecture that it should also be possible to obtain existence of a classical solution under the weaker assumption of $1$-regularity.
\end{remark}

\begin{proof}
Recall that $K$ is trivially extended over $\Omega\times (\R^N\backslash B_1)$. Fix $z,w\in \R^N$ such that $z\pm \frac{w}{2}\in \Omega$ and assume, without loss of generality, that $|w|\in (0,1/5)$. We claim that
\begin{align*}
 \int_{\R^N\sm B_{2|w|}} \left|\frac{K(z+\frac{w}{2},\xi+\frac{w}{2})}{|\xi+\frac{w}{2}|^N}-\frac{K(z-\frac{w}{2},\xi-\frac{w}{2})}{|\xi-\frac{w}{2}|^N}\right|\ell(|\xi|)d\xi \leq \L\ell(|w|).    
\end{align*}
Notice that the integrand over $B_{1+|w|/2}\sm B_{1-|w|/2}$ is bounded by a constant and the region has measure proportional to $|w|\leq \ell(|w|)$. So we focus on the region $B_{1-|w|/2}\sm B_{2|w|},$ where the kernel is of class $C^1$.

By the fundamental theorem of calculus and the hypothesis on the derivative of $K$,
\begin{align*}
&\left|\frac{K(z+\frac{w}{2},\xi+\frac{w}{2})}{|\xi+\frac{w}{2}|^N}-\frac{K(z-\frac{w}{2},\xi-\frac{w}{2})}{|\xi-\frac{w}{2}|^N}\right| \leq \int_{-1/2}^{1/2} \left|\frac{d}{dt}\1\frac{K(z+tw,\xi+tw)}{|\xi+tw|^N}\2\right|dt\\
&\leq \int_{-1/2}^{1/2} \frac{|D_xK(z+tw,\xi+tw) + D_yK(z+tw,\xi+tw)||w|}{|\xi+tw|^N}dt+N\int_{-1/2}^{1/2} \frac{|K(z+tw,\xi+tw)||w|}{|\xi+tw|^{N+1}}dt,\\
&\leq C|w||\xi|^{-(N+1)}.
\end{align*}
In the last inequality we have assumed that $\xi \in \R^N\sm  B_{2|w|}$ and $- t w \in B_{|w|}$ for $t\in(-1/2,1/2)$, such that $|\xi+tw|$ is comparable to $|\xi|$.

We bound in this way
\begin{align*}
\int_{B_{1-|w|/2} \sm B_{2|w|}} \left|\frac{K(z+\frac{w}{2},\xi+\frac{w}{2})}{|\xi+\frac{w}{2}|^N}-\frac{K(z-\frac{w}{2},\xi-\frac{w}{2})}{|\xi-\frac{w}{2}|^N}\right|\ell(|\xi|)d\xi &\leq C|w|\int_{B_{1-|w|/2} \sm B_{2|w|}} \frac{\ell(|\xi|)}{|\xi|^{N+1}}d\xi\\
&\leq C|w|\int_{2|w|}^\8 \frac{\ell(\r)}{\r^2}d\r.
\end{align*}
Finally, we use the semi-homogeneity of $\ell$ (Lemma \ref{prop1}) in the form $(|w|/\r)^{1/2} \leq \ell(|w|/\r) \leq c\ell(|w|)/\ell(\r)$ in order to complete the bound above with
\begin{align*}
C|w|\int_{2|w|}^\8 \frac{\ell(\r)}{\r^2}d\r \leq C|w|^{1/2}\ell(|w|)\int_{2|w|}^\8 \frac{d\r}{\r^{3/2}}d\r \leq C\ell(|w|),
\end{align*}
which concludes the lemma.
\end{proof}

\section{A priori estimates}\label{sec4}

In this section we show interior and boundary estimates. First, we show some maximum principles, then we revisit the interior $\cL^\a$-regularity from \cite{MR3626549}. By a bootstrapping argument we then prove $\cL^{1+\a}$ interior estimates which is above the regularity needed for classical solutions. Finally, we use the method of barriers to obtain estimates at the boundary of domains with the exterior sphere condition.

\subsection{Maximum principles}\label{mp:sec}

\begin{lemma}(Homogeneous maximum principle)\label{lem:mp} Let $\Omega\ss \R^N$ be an open bounded set, $K\in L^\infty(\Omega\times B_1)$, $L_K$ satisfy the uniform ellipticity condition \ref{itm:uellip} with parameters $0<\l\leq\L$, let $u\in L^1(B_1(\W)) \cap C^{Dini}_{loc}(\W) \cap C(\overline\W)$. If $u$ satisfies that
\[
L_Ku\leq 0 \text{ in }\W,\qquad 
u\leq 0 \text{ on } \R^N\sm\W,
\]
then $u\leq 0$ in $\W.$
\end{lemma}
\begin{proof}
Assume, by contradiction, that $M:= \sup_\W u >0$. By the hypotheses on $\W$ and $u$, there is $x_0 \in \{u=M\} \ss\W$ such that $(B_1(x_0) \sm B_{1/2}(x_0)) \sm \{u=M\}$ has positive measure. 

Consider $\varphi \in L^1(\R^N)\cap \cL^{1+\a}(B_{1/2}(x_0))$ such that
\[
\varphi(x)= 
\begin{cases}
M &\text{ if } x\in B_{1/2}(x_0),\\
u(x) &\text{ if } x\in \R^N\sm B_{1/2}(x_0).
\end{cases}
\]
This function $\varphi$ touches $u$ from above at $x_0$ and we reach a contradiction due to Lemma \ref{prop:elip}, because
\[
0\geq  L_Ku(x_0) \geq L_K \varphi(x_0) \geq \l \int_{B_1(x_0)\sm B_{1/2}(x_0)} \frac{M-u(y)}{|y-x_0|^N}dy > 0.
\]
\end{proof}

\begin{lemma}[Non-homogeneous maximum principle]\label{mp}
Let $\W\ss \R^N$ a bounded open set, $K\in L^\infty(\W \times B_1)$, $L_K$ satisfy the uniform ellipticity condition \ref{itm:uellip} with parameters $0<\l\leq \L$, and let $u \in L^1(B_1(\W))\cap C^{Dini}_{loc}(B_r)$. Then,
\[
\sup_{\W} u \leq \sup_{B_1(\W)\sm \W} u + C\|(L_K u)_+\|_{L^\8(\W)},
\]
for some $C=C(\l,\L,\diam(\W))>0$.
\end{lemma}

\begin{proof}
The case $\|(L_K u)_+\|_{L^\8(\W)}=0$ follows from Lemma \ref{lem:mp}. Assume that $\|(L_K u)_+\|_{L^\8(\Omega)}>0$ and consider
\begin{align*}
v:=\frac{u-\sup_{B_1(\W)\sm \W}u}{\|(L_K u)_+\|_{L^\8(\W)}}.
\end{align*}
We know that $L_K v \leq 1$ in $\W$ and $v\leq 0$ on $B_1(\W)\sm \W$.

Let $\a\geq 0$ to be fixed sufficiently large and $\varphi(x) := e^{-\a x_N}$. By the monotonicity of $\varphi$,
\begin{align*}
L_K\varphi(x) &\leq \int_{B_1\cap \{y_N<0\}} \frac{\varphi(x)-\varphi(x+y)}{|y|^N}\l dy + \int_{B_1\cap \{y_N\geq 0\}} \frac{\varphi(x)-\varphi(x+y)}{|y|^N}\L dy,\\
&= -e^{-\a x_N} \int_0^1 [\l(e^{\a y_N}-1)-\L(1-e^{-\a y_N})]h(y_N)dy_N,
\end{align*}
where, for $y>0$,
\[
h(y) := \int_{B_{(1-y^2)^{1/2}}^{N-1}} \frac{dz}{(|z|^2+y^2)^{N/2}} = \w_{N-1}y^{-1}\int_0^{(y^{-2}-1)^{1/2}} \frac{\r^{N-2}d\r}{(\r^2+1)^{N/2}} = O(y^{-1}) \text{ as $y\to0^+$}.
\]

By the convexity of the exponential, we have that, for $y\in(0,1)$,
\[
e^{\a y}-1 = \a\int_0^y e^{\a t}dt \geq  \a\int_0^y (1+\a t)dt = \a y + \tfrac{\a^2}{2} y^2
\]
and also $e^{-\a y}-1\geq -\a y$. Hence,
\begin{align*}
-e^{\a x_N}L_K\varphi(x) 
&\geq \int_0^1 [\l( \a y + \tfrac{\a^2}{2} y^2)+\L( -\a y)]h(y)dy\\
&= \1\tfrac{\l}{2}\int_0^1y^2h(y)dy\2\a^2 - \1(\L-\l)\int_0^1yh(y)dy\2\a.    
\end{align*}
So the right-hand side is positive for $\a$ sufficiently large. Once fixed $\a$, there is $c_0=c_0(\l,\L,N)>0$ such that $L_K\varphi \leq -c_0 e^{-\a R}$ in $\W$ provided that $\W \ss \{x_N\in[-R,R]\}$ for some $R>0$.

Finally, take $R$ even larger such that $B_1(\W) \ss \{x_N\in[-R,R]\}$ and consider
\[
w := v-\frac{e^{\a R}-\varphi}{c_0e^{-\a R}}.
\]
Then $w$ satisfies the hypotheses of the homogeneous maximum principle, from where we get $w\leq 0$ in $\W$ and therefore
\[
\frac{\sup_{\W} u - \sup_{B_1(\W)\sm \W} u}{\|(L_K u)_+\|_{L^\8(\W)}} \leq \frac{e^{2\a R}}{c_0}.
\]
\end{proof}

\begin{remark}\label{rmk:mp}
The exponential dependence of the constant in terms of the diameter of $\W$ is not optimal in the previous result. For $\W \ss B_R$ with $R\in(0,1/2)$, we have that $\varphi = \chi_{B_R}$ is an upper barrier with $L_K\varphi \geq c|\ln R|$ in $B_R$. This means that for such domains the constant can be improved to be of order $\ell(R)$. We expect this to be the asymptotic behavior as $R\to0$ of $\max_{B_R} u$, where $u$ is the torsion function, i.e. the solution of
\[
\begin{cases}
L_K u = 1 \text{ in } B_R,\\
u = 0 \text{ on } B_1\sm B_R.
\end{cases}
\]
On the other hand, we do not have a good guess of what should be the asymptotic behavior of $\max_{B_R} u$ as $R\to\8$.
\end{remark}

\begin{remark}\label{othermp}
The previous maximum principle (Lemma \ref{mp}) also applies whenever $K \in L^\8(\W\times \R^N)$ is elliptic in the sense that $K\geq0$ in $\W\times \R^N$ and uniformly elliptic in the sense of \ref{itm:uellip}, a condition imposed on $K(x,y)$ with $y \in B_1$. The proofs are the same. We use this fact in the proof of Theorem \ref{thm:loglap} to apply an approximation argument which allows us to use the results in \cite{feulefack2021nonlocal}.
\end{remark}

\begin{remark}
Other maximum principles for operators with zero order kernels can be found in \cite{MR3995092,feulefack2021nonlocal,FKT20}.  In particular, in \cite{FKT20} the classification of nonnegative finite energy solutions of a nonlinear problem involving an operator on the sphere with a zero order kernel is studied. Due to the lack of a regularity theory for those solutions, the authors work in a weak setting arguing via symmetry and conformal invariance. In particular, in \cite[Section 3]{FKT20}, the authors show some maximum principles for antisymmetric functions, including a small volume maximum principle and a strong maximum principle.
\end{remark}

\subsection{Lower order estimates}\label{km}

In this section we show the following.
\begin{theorem}\label{thm1}
Let $\W\ss\R^N$ be an open set with bounded inner radius $R:=\sup_{x\in\W}\dist(x,\p\W)<\8$, $K\in L^\infty(\Omega\times B_1)$, $L_K$ satisfy the uniform ellipticity condition \ref{itm:uellip} with parameters $0<\l\leq \L$, and let $u \in L^\8(B_1(\W))\cap C^{Dini}_{loc}(\W)$ be such that $\|L_Ku\|^{(1)}_{\mathcal L^0(\W)}<\8$.  Then there is $\a=\a(\l,\L,N)\in(0,1)$ such that $u\in \cL^\a_{loc}(\W)$ and
\[
\|u\|^{(0)}_{\mathcal L^\a(\W)} \leq C\left(\|u\|_{L^\infty(B_1(\W))}+\|L_Ku\|^{(1)}_{\mathcal L^0(\W)}\right)
\]
for some $C=C(\l,\L,N,R)>0$.
\end{theorem}

This result is essentially due to \cite{MR3626549}.  The proof we present below is closely related to \cite{MR3626549} but we have some slightly different assumptions on the operator $L_K$; for instance, we do not assume that the kernel $K(x,y)$ is even in $y$. Finally notice that the regularity estimate on the solution is independent of the regularity of the kernel $K=K(x,y)$, which is merely assumed to be uniformly bounded from above and away form zero.

The proof of Theorem \ref{thm1} relies on the following diminish of oscillation lemma.

\begin{lemma}\label{meas_est}
There exist $r_0,\a,\d\in(0,1)$ depending on $0<\l\leq \L$ and $N$ such that the following holds: Let $\r\in(0,r_0)$, 
let $K\in L^\infty(B_r\times B_1)$, $L_K$ satisfy the uniform ellipticity condition \ref{itm:uellip} with parameters $0<\l\leq \L$, and let $u \in L^1(B_{1+2\r^2})\cap C^{Dini}_{loc}(B_{2\r^2})$ be such that
\begin{align*}
&L_K u \geq -\d\ell^{\a-1}(\r^2) \text{ in } B_{2\r^2},\\
&u(x) \geq -t_\a(x) := -(\ell^\a(|x|)-\ell^\a(\r))_+ \text{ for all } x \in B_{1+2\r^2},\\
&\mu(\{u\geq \ell^\a(\r)/2\}\cap B_{\r} \sm B_{3\r^2}) \geq \tfrac{1}{2}\mu\1B_{\r} \sm B_{3\r^2}\2, \qquad\text{where }\m(dx) :=\ell(|x|) |x|^{-N}dx.
\end{align*}
Then
\[
\inf_{B_{\r^2}} u \geq \ell^\a(\r)-\ell^\a(\r^2).
\]
\end{lemma}

In the remainder of the section we assume that $K\in L^\infty(B_r\times B_1)$ and that $L_K$ satisfies the uniform ellipticity condition \ref{itm:uellip} with parameters $0<\l\leq \L$.

Let $\b \in C^\8_0(B_1)$ be a fixed bump function taking values in $[0,1]$ and being identically equal to one in $B_{1/2}$. Moreover, let
\begin{align}\label{betafun}
\b_r(x) := \b(r^{-1}x) \qquad \text{ for }r>0 \text{ and }x\in \R^N.
\end{align}
The idea of the proof of Lemma \ref{meas_est} is to fit the following barrier below the super-solution
\begin{align*}
\varphi := \underbrace{(\ell^\a(\r)-\ell^\a(\r^2))\b_{2\r^2}}_{\text{Bump function}} + \underbrace{\tfrac{\ell^\a(\r)}{2}\chi_{\{u\geq \ell^\a(\r)/2\}\cap B_{\r} \sm B_{3\r^2}}}_{\text{Gaining term}} - \underbrace{t_\a}_{\text{Tail control}}.
\end{align*}

The first term is a \emph{bump function} supported in $B_{2\r^2}$ which provides the desired improvement from below in $B_{\r^2}$. The second term is the one that uses the measure hypothesis in the lemma and gives us then some room to fit the estimates for the barrier, because of this favorable behavior we call it the \emph{gaining or good term}. The last term is common in the regularity theory for non-local operators and it reflects the fact that in the standard iterative arguments one should always keep control of the \emph{tails} of the solutions.

Notice that $\varphi$ is smooth in $B_{2\r^2}$ where we can compute $L_K$. We split in this way the computation of $L_K\varphi$ in three terms for which we give the details in the following lemmas.

\begin{lemma}\label{lem_bump}
Let $\b$ and $\b_r$ as in \eqref{betafun}. Then, there are $r_0\in(0,1)$ and $C>0$ depending on $N$ and $\b$ such that
\[
L_K\b_r \leq C\L\ell^{-1}(r)\text{ in }B_{r}\qquad \text{ for all }r\in(0,r_0).
\]
\end{lemma}

\begin{proof}
For $x = r\bar x \in B_{r}$,
\begin{align*}
L_K\b_r(x)
&= \int_{B_{1}(x)}\frac{\b_r(x)-\b_r(y)}{|y-x|^N}K(x,y-x)dy,\\
&= \int_{B_{1}(x)}\frac{\b(r^{-1}x)-\b(r^{-1}y)}{|y-x|^N}K(x,y-x)dy,\\
&= \int_{B_{r^{-1}}(\bar x)}\frac{\b(\bar x)-\b(\bar y)}{|\bar y-\bar x|^N}K(r\bar x,r(\bar y-\bar x))d\bar y\\
&\leq \L[\b]_{C^{0,1}(B_2)}\int_{B_2(\bar x)}\frac{1}{|\bar y-\bar x|^{N-1}}d\bar y
+\L\int_{B_{r^{-1}}(\bar x)\backslash B_2(\bar x)}\frac{1}{|\bar y-\bar x|^N}d\bar y\\
&\leq C\L([\b]_{C^{0,1}(B_1)}+ |\ln r|),\\
&\leq C\L\ell^{-1}(r).
\end{align*}
In the last inequality we need to assume that $r_0$ is sufficiently small in order to absorb the constant term.
\end{proof}

\begin{lemma}\label{lem:gain}
Let $\r\in(0,1/3)$, $A\ss B_\r \sm B_{3\r^2}$, and $\m(dx) = \ell(|x|)|x|^{-N}dx$. Then it holds that $L_K\chi_A \leq -c\l\ell^{-1}(\r)\m(A)$ in $B_{2\r^2}$ for some $c>0$.
\end{lemma}

\begin{proof}
There is $c>0$ such that, for all $x\in B_{2\r^2}$,
\begin{align*}
L_K\chi_A(x) &\leq -\l\int_A \frac{dy}{|y-x|^N}\leq -c\l\ell^{-1}(\r)\int_A \frac{\ell(|y|)}{|y|^N}dy =  -c\l\ell^{-1}(\r)\m(A),
\end{align*}
where we used that $|y-x|$ is comparable to $|y|$, because $x\in B_{2\r^2}$ and $y\in A\subset B_\rho\sm B_{3\r^2}$.
\end{proof}

\begin{lemma}\label{lem:tail}
Let $\r\in(0,1)$, $\a\in(0,1/2)$, and $t_\a(x) = (\ell^\a(|x|)-\ell^\a(\r))_+$. Then it holds that $L_K t_\a \geq -C\L(1+\a\ell^{\a-1}(\r))$ in $B_{\r/2}$ for some $C>0$.
\end{lemma}
\begin{proof}
For $x\in B_{\r/2}$
\begin{align*}
L_K t_\a(x) &\geq -\L\int_{B_1(x)\sm B_\r} \frac{\ell^\a(|y|)-\ell^\a(\r)}{|y-x|^N}dy \geq -C\L\11+\int_{B_{\r_0} \sm B_\r} \frac{|\ln|y||^{-\a}-|\ln \r|^{-\a}}{|y|^N}dy\2,
\end{align*}
where the last inequality is justified by using that $|y-x|$ is comparable to $|y|$ (because $x\in B_{\r/2}$ and $y\in \R^N\sm B_{\r}$) and truncating the integral at the radius $\r_0$ in order to replace $\ell(|y|)$ by $|\ln|y||^{-1}$. Finally, we see that the integral is computed by
\[
\int_{B_{\r_0} \sm B_\r} \frac{|\ln|y||^{-\a}-|\ln \r|^{-\a}}{|y|^N}dy \leq \w_N\int_0^{|\ln \r|} (z^{-\a}-|\ln \r|^{-\a})dz = \frac{\w_N\a}{1-\a}|\ln \r|^{1-\a}.
\]
\end{proof}

\begin{proof}[Proof of Lemma \ref{meas_est}]
Let $\b \in C^\8_0(B_1)$ be the bump function from Lemma \ref{lem_bump}, 
\begin{align*}
A:=\{u\leq \ell^\a(\r)/2\}\cap B_\r\sm B_{3\r^2},\qquad t_\a(x):= (\ell^\a(|x|)-\ell^\a(\r))_+,   
\end{align*}
and construct from these ingredients the barrier
\begin{align*}
\varphi := (\ell^\a(\r)-\ell^\a(\r^2))\b_{2\r^2} + \tfrac{\ell^\a(\r)}{2}\chi_A - t_\a.    
\end{align*}
Our goal is to show that, for $r_0,\a,\d\in(0,1)$ small, it holds that $L_K\varphi \leq -\d\ell^{\a-1}(\r^2)$ in $B_{2\r^2}$. Once this is established, we would obtain that $u\geq \varphi$ in $B_{2\r^2}$ by applying the maximum principle to $\varphi-u$ in $B_{2\r^2}$. Hence, $u \geq \varphi = \ell^\a(\r)-\ell^\a(\r^2)$ in $B_{\r^2}$, which would settle the proof.

Using lemmas \ref{lem_bump}, \ref{lem:gain}, and \ref{lem:tail},  we have, for $x\in B_{2\r^2}$, that
\begin{align*}
  L_K\varphi(x) &\leq (\ell^\a(\r)-\ell^\a(\r^2))C\ell^{-1}(2\r^2) - c\ell^{\a-1}(\r)\m(A) + C(1+\a\ell^{\a-1}(\r)),\\
  &\leq (2^\a-1)C\ell^{\a-1}(\r^2) - c\ell^{\a-1}(\r^2)\m(A) + C(1+\a\ell^{\a-1}(\r^2)),\\
  &\leq C\ell^{\a-1}(\r^2)(2^\a-1 + \a-c)+C.
\end{align*}
For the second (gaining) term we have used that
\[
\m(A) \geq \frac{1}{2}\m(B_\r\sm B_{3\r^2}) = \frac{\w_N}{2}\int_{3\r^2}^\r \frac{\ell(t)}{t}dt \geq \frac{\w_N}{2}\int_{3\r^2}^{\sqrt 3\r} \frac{d(\ln t)}{|\ln t|} = \frac{\w_N\ln 2}{2}>0.
\]
Using now that $2^\a-1+\a\to 0$ as $\a\to 0$ we have that there is some $\alpha_0=\alpha_0(N)>0$ such that, for $\a\in(0,\alpha_0)$, we can to bound the right-hand side by $-\d\ell^{\a-1}(\r^2)$ as long as $\r\in(0,r_0)$, with $r_0$ sufficiently small in order to absorb the positive constant.
\end{proof}

\begin{corollary}\label{cor:mod_cont}
Let $r_0,\a\in(0,1)$ be as in Lemma \ref{meas_est}. Given $r\in(0,r_0)$ and $u \in L^\8(B_{1+r})\cap C^{Dini}_{loc}(B_r)$ with $L_Ku \in L^\8(B_r)$, then
\[
\ell^{\a}(r)\sup_{\r\in (0,r)}\ell^{-\a}(\r) [u]_{\mathcal L^0(B_{\r})} \leq C\left([u]_{\mathcal L^0(B_{1+r})}+\ell(r)\|L_Ku\|_{L^\8(B_r)}\right)
\]
for some $C>0$ depending on $\l$, $\L$, and $N$.
\end{corollary}
\begin{proof}
Let $r\in(0,r_0)$ and let $\delta\in(0,1)$ be as in Lemma \ref{meas_est} and
\[
v:=\frac{\ell^\a(r^2) \d }{[u]_{\mathcal L^0(B_{1+r})}+
\ell(r)\|L_Ku\|_{L^\8(B_r)}}u.
\]
It suffices to show that $[v]_{\cL^0(B_\r)} \leq \ell^\a(\r)$ for $\r\in(0,r)$.

The function $v$ is constructed such that $|L_Kv| \leq \d \ell^{\a-1}(r)$ in $B_r$ and $[v]_{\mathcal L^0(B_\r)} \leq \ell^\a(\r)$ for $\r\in[r^2,r)$. Next, we show that $[v]_{\mathcal L^0(B_\r)} \leq \ell^\a(\rho)$ also holds for $\r \in(0,r^2)$.

Assume, by contradiction, that there are $\r_1 \in (0,r^2)$ and $\r_2 \in (\r_1^2,\r_1)$ such that
\begin{align*}
\sup_{\r\in(\r_1,1+r)}\ell^{-\a}(\r)\osc_{B_{\r}}  v \leq 1 \qquad\text{ and } \qquad \osc_{B_{\r_2}}  v > \ell^{\a}(\r_2).
\end{align*}
Let $\r:=\r_2^{1/2}$, $m:= \inf_{B_{\r}} v$, and consider the sets
\begin{align*}
S_+ := \{v - m \geq \ell^\a(\r)/2\}\cap B_{\r} \sm B_{3\r^2},\qquad S_- := \{v - m \leq \ell^\a(\r)/2\}\cap B_{\r} \sm B_{3\r^2}.
\end{align*}
For $\mu$ being the measure defined in Lemma \ref{meas_est}, either $\mu(S_-)$ or $\mu(S_+)$ is at least $\frac{1}{2}\mu(B_{\r} \sm B_{3\r^2})$. Assume, without loss of generality, that $\mu(S_+) \geq \frac{1}{2}\mu(B_{\r} \sm B_{3\r^2})$ and consider then the translation $w := v - m.$ The assumption on $S_+$ means that
\[
\m(\{w\geq \ell^\a(\r)/2\}\cap B_{\r} \sm B_{3\r^2}) = \m(S_+) \geq \tfrac{1}{2}\m(B_{\r} \sm B_{3\r^2}). 
\]
Notice that we still have $L_Kw\geq -\d \ell^{\a-1}(2\r^2)$ in $B_{2\r^2}$ and moreover, since $\r>\r_1$, we get by the oscillation hypothesis that $w(x) \geq t_\a(x) = -(\ell^\a(|x|)-\ell^\a(\r))_+$.

Given that $w$ satisfies the conditions of Lemma \ref{meas_est}, we reach a contradiction with the hypothesis $\osc_{B_{\r_2}} v>\ell^\a(\r_2)$, because
\begin{align*}
\osc_{B_{\r_2}} v = \osc_{B_{\r^2}} w \leq \ell^\a(\r)-(\ell^\a(\r)-\ell^\a(\r^2)) = \ell^\a(\r_2).
\end{align*}
This concludes the proof.
\end{proof}

We are ready to show Theorem \ref{thm1}

\begin{proof}[Proof of Theorem \ref{thm1}]
The proof follows from Corollary \ref{cor:mod_cont} using a standard covering argument.  We give the details for completeness. 
For any $B_{2r}(x) \ss \W$ we have that
\begin{align*}
    [u]_{\mathcal L^\a(B_r(x))} \leq \sup_{y\in B_r(x)} \sup_{\r\in(0,2r)} \ell^{-\a}(\r)[u]_{\mathcal L^0(B_\r)}.
\end{align*}
Hence by Corollary \ref{cor:mod_cont} we get that as long as $r\in(0,r_0)$ and $B_{2r}(x)\ss\W$
\begin{align*}
\ell^\a(r)[u]_{\mathcal L^\a(B_r(x))} \leq C\1\|u\|_{L^\8(B_{1}(\W))} +\ell(r)\|L_Ku\|_{L^\8(B_{2r}(x))}\2.
\end{align*}
By Lemma \ref{prop2} we get that the second term gets bounded by $\|L_Ku\|_{\mathcal L^0(\W)}^{(1)}$ as expected. By adding $\|u\|_{L^\8(B_r(x))}$ on both sides and applying once again Lemma \ref{prop2} we conclude the proof.
\end{proof}

\subsection{Higher order estimates}\label{hoe}

The goal of this section is to show the following result.
\begin{theorem}\label{thm2}
Let $\W$ be an open set with bounded inner radius $R:=\sup_{x\in\W}\dist(x,\p\W)<\8$. Let $L_K$ be translation invariant \ref{itm:ti}, uniformly elliptic \ref{itm:uellip} with parameters $0<\l\leq \L$, and $1$-regular (Definition \ref{def:reg}) with constant $\L$. Let $u\in L^\8(B_1(\W))\cap C^{Dini}_{loc}(\W)$ be such that $\|L_Ku\|^{(1)}_{\mathcal L^1(\W)}<\8$.  Then there exists $\a=\a(\l,\L,N)\in(0,1)$ such that $u\in \mathcal L^{1+\a}_{loc}(\W)$ and
\[
\|u\|^{(0)}_{\mathcal L^{1+\a}(\W)} \leq C\1\|u\|_{L^\8(B_1(\W))}+\|L_Ku\|^{(1)}_{\mathcal L^1(\W)}\2
\]
for some $C=C(\l,\L,N,R)>0$.
\end{theorem}

For the proof, we show first some auxiliary lemmas.

\subsubsection{An extrapolation property for the log-Hölder moduli of continuity}

In this subsection we consider one dimensional functions $u:(-\r,\r)\to\R$. We denote the first and second order differences of size $\e>0$ as
\[
\d_\e u(x) := u(x+\e)-u(x), \qquad \d_\e^2 u(x) = u(x+2\e)-2u(x+\e)+u(x).
\]
Notice that $\d_\e^2u(x)$ is the second order difference centered at $x+\e$ and not $x$.

The following lemma extends \cite[Lemma 5.6]{MR1351007} to the log-Hölder setting. To simplify the notation we remove one parenthesis whenever we refer to different function spaces over open intervals, such as $C(a,b) := C((a,b))$.

\begin{lemma}\label{lem:stein}
Given $\a>0$ there exists $C>0$ such that, for $\r>0$ and $u\in C(-\r,\r)$, it holds that
\[
\ell^\a(\r) \sup_{\e\in (0,\r)} \ell^{-\a}(\e)\|\d_\e u\|_{L^\8(-\r,\r-\e)} \leq C\1[u]_{\mathcal L^{0}(-\r,\r)} +\ell^\a(\r)\sup_{\e\in(0,\r)}\ell^{-\a}(\e)\|\d_\e^2 u\|_{L^\8(-\r,\r-2\e)}\2.
\]
\end{lemma}
\begin{proof}
Let $\rho>0$, $u\in C(-\r,\r)$, and assume, without loss of generality, that
\begin{align}\label{1a1}
[u]_{\mathcal L^{0}(-\r,\r)} +\ell^\a(\r)\sup_{\e\in(0,\r)}\ell^{-\a}(\e)\|\d_\e^2 u\|_{L^\8(-\r,\r-2\e)}\leq 1.        
\end{align}
Our goal is then to show that, for some $C>0$ and for all $\eps\in (0,\rho)$,
\begin{align}\label{1c1}
\ell^{-\a}(\e)\|\d_\e u\|_{L^\infty(-\rho,0]}\leq C\ell^{-\a}(\r).
\end{align}
The estimate on the remaining interval $[0,\r-\e)$ follows by considering a symmetric reasoning.

Fix $x\in (-\rho,0)$, $\eps\in (0,\rho)$, and let $i$ be a non-negative integer such that $x+2^i\e < \r \leq x+2^{i+1}\e$. By the definition of the finite difference,
\begin{align*}
\d_{2^{i-j}\e}u(x) - 2\d_{2^{i-1-j}\e} u(x) =\d_{2^{i-1-j}\e}^2 u(x)\qquad \text{ for } j=0,\ldots,i-1.
\end{align*}
Moreover, by \eqref{1a1},
\begin{align*}
&|\d_{2^i\e} u(x)|=|u(x+2^i\e)-u(x)|\leq [u]_{\cL^0(-\r,\r)}\leq 1,\\
&|\d_{2^{i-1-j}\e}^2 u(x)| \leq \|\d_{2^{i-1-j}\e}^2 u\|_{L^\8(-\r,\r-2\e)}\leq \ell^{-\a}(\r)\ell^{\a}(2^{i-1-j}\e).
\end{align*}

Using a telescopic sum and keeping in mind that, by the construction of $i$ and Lemma \ref{prop1}, $2^{-i}\leq 2\eps\rho^{-1}\leq 2\ell^\a(\eps \rho^{-1})\leq c\ell^{\alpha}(\eps)\ell^{-\alpha}(\rho)$ then
\begin{align*}
|\d_\e u(x)| &\leq 2^{-i}(|\d_{2^i\e}u(x)|+|\d_{2^i\e}u(x) - 2^i\d_\e u(x)|)\\
&\leq C\1\ell^\a(\e)\ell^{-\a}(\r) + \sum_{j=0}^{i-1} 2^{j-i}|\d_{2^{i-j}\e}u(x) - 2\d_{2^{i-1-j}\e} u(x)|\2\\
&\leq C\1 \ell^\a(\e)\ell^{-\a}(\r) + \sum_{j=0}^{i-1} 2^{j-i}|\d_{2^{i-1-j}\e}^2 u(x)|\2\\
&\leq C\ell^{-\a}(\r)\1 \ell^\a(\e) + \sum_{j=0}^{i-1} 2^{j-i}\ell^\a(2^{i-1-j}\e)\2\\
&\leq C\ell^\a(\e)\ell^{-\a}(\r) \1 1+ \sum_{j=1}^{i} 2^{-j} \ell^{-\a}(2^{-j})\2.
\end{align*}
Since, for $j$ sufficiently large, $\ell^{-\a}(2^{-j})=|\ln(2^{-j})|^\alpha=j^\alpha |\ln 2|^\alpha$ we have that the sum is uniformly bounded and \eqref{1c1} follows.
\end{proof}

\subsubsection{Log-Hölder bootstrap}

In this section we consider again $u:B_1(\W)\ss\R^N\to \R$ and denote the first and second order differences of size $\e>0$ and direction $e\in\p B_1$ by
\[
\d_{\e e} u(x) := u(x+\e e) - u(x), \qquad \d_{\e e}^2 u(x) := u(x+2\e e)-2u(x+\e e)+u(x).
\]

For local equations one can improve the regularity of the solution by considering difference quotients and iterating the lower order estimates. However, for non-local equations the tails of the solutions do not have an improvement in regularity and one must modify the strategy. The idea is then to multiply the function by a smooth cut-off function whose purpose is to ``hide the irregular part of the solution". The trade off we must pay for this step is that the equation for the difference quotient has an additional term in the right-hand side that depends on the cut-off. Nevertheless, the regularity of this term is guaranteed by a regularity assumption on the kernel.

Recall the definition of $\b$-regularity given in Definition \ref{def:reg}.

\begin{lemma}\label{lem:lin2}
Let $r,\b\in(0,1]$, $K\in L^\infty(B_1)$, $L_K$ be a $\b$-regular, operator satisfying the translation invariance condition \ref{itm:ti}, and the uniform ellipticity condition \ref{itm:uellip} with parameters $0<\l\leq\L$. Let $\eta \in C^\8_0(B_{5r})$ be such that $\eta=1$ in $B_{4r}$, $\e\in(0, r)$, and $e\in \p B_1$. If $u \in L^\8(B_{1+2r})\cap C^{Dini}_{loc}(B_{2r})$, then
\[
\|L_K \d_{\e e}(\eta u) - \d_{\e e}L_Ku\|_{L^\8(B_r)} \leq C\ell^\b(\e)\ell^{-1}(r)\|u\|_{L^\8(B_{1+2r})},
\]
for some $C>0$ depending on $\l$, $\L$, $N$, and the $\b$-regularity parameter of $L_K$.
\end{lemma}

\begin{proof}
Without loss of generality, assume that $K = 0$ in $\R^N\sm B_1$. Let $x\in B_{r}$ be such that $x+\e e \in B_{2r}$.  Then, by the definition of $\eta$,
\begin{align*}
    L_K(\d_{\e e}(\eta u))(x) &= \int_{B_1(x)} \frac{u(x+\e e)-u(x)-(\eta u)(y+\e e)+(\eta u)(y)}{|y-x|^N}K(y-x)dy,\\
    &=  \int_{B_1(x)} \frac{u(x+\e e)-(\eta u)(y+\e e)}{|y-x|^N}K(y-x)dy - \int_{B_1(x)} \frac{u(x)-(\eta u)(y)}{|y-x|^N}K(y-x)dy.
\end{align*}
Using that $(1-\eta)$ vanishes in a neighborhood of $x$ and $x+\e e \in B_{2r}$, we have that
\begin{align}
    L_K\d_{\e e}(\eta u)(x) &= L_Ku(x+\e e) - L_Ku(x) + \int_{B_1(x)} \frac{((1-\eta)u)(y+\e e)-((1-\eta)u)(y)}{|y-x|^N}K(y-x)dy\notag\\
    &= \d_{\e e}L_Ku(x) + \int_{\R^N} ((1-\eta)u)(y)\1\frac{K(y-x-\e e)}{|y-x-\e e|^N} - \frac{K(y-x)}{|y-x|^N}\2dy\notag\\
    &= \d_{\e e}L_Ku(x) + \int_{B_{1+2r}} u(y)\bar K(x,y)dy,\label{eq2}
\end{align}
where 
\[
\bar K(x,y):= (1-\eta(y))\1\frac{K(y-x-\e e)}{|y-x-\e e|^N} - \frac{K(y-x)}{|y-x|^N}\2,\qquad x\in B_r,\ y\in \R^N.
\]
By the $\b$-regularity of $L_K$, we have that for any $x\in B_r$ (and using the change of variables $\xi=y-(x+(\e/2)e)$),
\begin{align}
\int_{\R^N} |\bar K(x,y)|dy
&\leq \int_{\R^N\backslash B_{4r}} \left|\frac{K(y-x-\e e)}{|y-x-\e e|^N} - \frac{K(y-x)}{|y-x|^N}\right|dy\notag
\\
&\leq C\ell^{-1}(r)\int_{\R^N\sm B_{2r}} 
\left|\frac{K(\xi-\frac{\e}{2}e)}{|\xi-\frac{\e}{2}e|^N} - \frac{K(\xi+\frac{\e}{2}e)}{|\xi+\frac{\e}{2}e|^N}\right| \ell(|\xi|)d\xi
\notag\\
&\leq C\ell^{-1}(r)\int_{\R^N\sm B_{2\e}} 
\left|\frac{K(\xi-\frac{\e}{2}e)}{|\xi-\frac{\e}{2}e|^N} - \frac{K(\xi+\frac{\e}{2}e)}{|\xi+\frac{\e}{2}e|^N}\right| \ell(|\xi|)d\xi \leq C\ell^\b(\e)\ell^{-1}(r).\label{eq3}
\end{align}
Therefore, by \eqref{eq2} and \eqref{eq3},
\begin{align*}
|L_K \d_{\e e}(\eta u)(x) - \d_{\e e}L_Ku(x)| \leq C\ell^\b(\e)\ell^{-1}(r)\|u\|_{L^\8(B_{1+2r})}
\end{align*}
for $x\in B_r$, as claimed.
\end{proof}

\begin{corollary}\label{cor1}
Let $\beta\in(0,1]$, $L_K$ be a translation invariant, $\b$-regular, uniformly elliptic operator with parameters $\l$ and $\L$, $r\in(0,r_0)$. There exist $\a,r_0\in(0,1)$ depending on $0<\l\leq\L$, and $N$ such that the following holds: Given $u \in L^\8(B_{1+2r})\cap C^{Dini}_{loc}(B_{2r})$, $e\in \p B_1$ and $\e\in(0, r/2)$ with $\d_{\e e}(L_Ku) \in L^\8(B_{2r})$, then
\begin{align*}
&\ell^{\a+\b}(r)\left\|\frac{\d_{\e e}^2u}{\ell^{\a+\b}(\e)}\right\|_{L^\8(B_{r/2})}\\
&\leq C\1\|u\|_{L^\8(B_{1+2r})}+\ell^\b(r)\left\|\frac{\d_{\e e}u}{\ell^{\b}(\e)}\right\|_{L^\8(B_{6r})}
 +\ell^{1+\b}(r)\left\|\frac{\d_{\e e}(L_Ku)}{\ell^\b(\e)}\right\|_{L^\8(B_{r})}
  \2
\end{align*}
for some $C>0$ depending on $\l$, $\L$, $N$, and the $\b$-regularity parameter of $L_K$.
\end{corollary}

\begin{proof}

Let $\eta \in C^\8_0(B_{5r})$ be such that $\eta=1$ in $B_{4r}$ and $\|D\eta\|_{L^\8} \leq Cr^{-1}$. Notice first that for any $\a \in(0,1)$ and $\e\in(0,r/2)$
\[
\left\|\frac{\d_{\e e}^2 u}{\ell^\a(\e)}\right\|_{L^\8(B_{r/2})} \leq \sup_{\substack{x\in B_{r/2}\\\r\in (0,r/2)}} \ell^{-\a}(\r)[\d_{\e e}(\eta u)]_{\mathcal L^0(B_\r(x))}.
\]
To check this identity one just has to set $\r=\e$ and compute the oscillation between the center of the ball and its boundary.

We now choose $\a\in(0,1)$ according to Corollary \ref{cor:mod_cont} such that for any $x\in B_{r/2}$
\begin{align*}
\ell^{\a}(r)\sup_{\r\in (0,r/2)}\ell^{-\a}(\r) [\d_{\e e}(\eta u)]_{\mathcal L^0(B_{\r}(x))} 
&\leq C\left(\|\d_{\e e}(\eta u)\|_{L^\8(B_{1+r})}+\ell(r)\|L_K \d_{\e e}(\eta u)\|_{L^\8(B_{r})}\right).
\end{align*}

By Lemma \ref{lem:lin2},
\begin{align*}
\|L_K \d_{\e e}(\eta u)\|_{L^\infty(B_{r})} \leq C\ell^\b(\e)\ell^{-1}(r)\|u\|_{L^\8(B_{1+2r})}+\|\d_{\e e}L_Ku\|_{L^\infty(B_{r})}.
\end{align*}
By the discrete product rule
\begin{align*}
\|\d_{\e e}(\eta u)\|_{L^\8(B_{1+r})} &\leq \sup_{x \in B_{1+r}}|\eta(x+\e e)(u(x+\e e) - u(x))| + C\e r^{-1}\|u\|_{L^\8(B_{1+r})},\\
&\leq \|\d_{\e e}u\|_{L^\8(B_{6r})} + C\ell^{\b}(\e)\ell^{-\b}(r)\|u\|_{L^\8(B_{1+r})}.
\end{align*}
By putting together these inequalities we obtain the desired estimate.
\end{proof}

By combining this corollary with the extrapolation Lemma \ref{lem:stein} we obtain the following estimate whenever the right-hand side $L_Ku \in \mathcal L^\b(B_{4r})$ with $\b\in(0,1]$.

\begin{corollary}\label{cor2}
Let $\b\in(0,1]$, $L_K$ be a translation invariant, $\b$-regular, uniformly elliptic operator with parameters $0<\l\leq \L$, $r\in(0,r_0)$. There exist $r_0,\a\in(0,1)$ depending on $0<\l\leq\L$, and $N$ such that the following holds: Given $u \in L^\8(B_{1+r})\cap \mathcal L^{\b}(B_{r})$ with $L_Ku \in \cL^\b(B_{r})$, then $u \in \mathcal L^{\b+\a}(B_{r/10})$ with the estimate
\begin{align*}
\ell^{\b+\a}(r)[u]_{\mathcal L^{\b+\a}(B_{r/60})}
\leq C(\|u\|_{L^\8(B_{1+r})}+
\ell^\b(r)[u]_{\mathcal L^\b(B_{r})}
+\ell^{1+\b}(r)[L_Ku]_{\mathcal L^\b(B_{r})}
)
\end{align*}
for some $C>0$ depending on $\l$, $\L$, $N$, and the $\b$-regularity parameter of $L_K$.
\end{corollary}

\begin{proof}
Let $x,y\in B_{r/10}$ with $x\neq y$, $e:=\frac{y-x}{|y-x|}$, and let $v(\tau):=u(x+\tau e)$ such that for $\t\in (-r/5,r/5)$ it holds that $x+\t e\in B_{r/2}$. In particular, $y = x+\t e$ for $\t = |y-x| \in (0,r/5)$. By Lemma \ref{lem:stein}
\begin{align*}
\ell^{\b+\a}(r)\frac{|u(x)-u(y)|}{\ell^{\a+\b}(|x-y|)}
&\leq  C\ell^{\b+\a}(r)\sup_{\e\in (0,r/5)} \left\|\frac{\d_\e v}{\ell^{\b+\a}(\e)}\right\|_{L^\8(-r/5,r/5-\e)},\\
&\leq C\1 [v]_{\mathcal L^{0}(-r/5,r/5)} +\ell^{\b+\a}(r)\sup_{\e\in(0,r/5)}\left\|\frac{\d_{\e}^2 v}{\ell^{\b+\a}(\e)}\right\|_{L^\8(-r/5,r/5-2\e)}\2,\\
&\leq C\1 [u]_{\mathcal L^{0}(B_{r/2})} +\ell^{\b+\a}(r)\sup_{\substack{e\in \p B_1\\\e\in(0,r/5)}}\left\|\frac{\d_{\e e}^2 u}{\ell^{\b+\a}(\e)}\right\|_{L^\8(B_{r/2})}\2.
\end{align*}
We conclude the estimate by applying now Corollary \ref{cor1} and using that $u \in \mathcal L^\b(B_{r})$ and $L_Ku \in \mathcal L^\b(B_{r})$ to bound the incremental quotients in the right-hand side.
\end{proof}

Now we iterate this result to obtain higher regularity.

\begin{corollary}\label{cor:high}
There exist $r_0,\a,c\in(0,1)$ depending on $0<\l\leq\L$, and $N$ such that the following holds: Let $L_K$ be a translation invariant, $1$-regular, uniformly elliptic operator with parameters $\l$ and $\L$, and $r\in(0,r_0)$. Given $u \in L^\8(B_{1+r})\cap C^{Dini}_{loc}(B_{r})$ such that $L_Ku \in \cL^1(B_{r})$, then $u \in \mathcal L^{1+\a}(B_{cr})$ with the estimate
\[
\ell^{1+\a}(r)[u]_{\mathcal L^{1+\a}(B_{cr})} \leq C\1\|u\|_{L^\8(B_{1+r})}+\|L_Ku\|^{(1)}_{\mathcal L^0(B_{r})}+\ell^{2}(r)[L_Ku]_{\mathcal L^1(B_{r})}\2
\]
for some $C>0$ depending on $\l$, $\L$, $N$, and the $\b$-regularity parameter of $L_K$.
\end{corollary}
\begin{proof}
In the following, $c\in(0,1)$ and $C>0$ are possibly different constants independent of $u$.

By Lemma \ref{prop2} and Theorem \ref{thm1}, there is $\alpha\in(0,1)$ such that
\begin{align}\label{eq8}
\ell^{\a}(r)[u]_{\mathcal L^{\a}(B_{r/2})} \leq C[u]^{(0)}_{\mathcal L^\a(B_{r})} \leq C\left(\|u\|_{\mathcal L^\infty(B_{1+r})}+\|L_Ku\|^{(1)}_{\mathcal L^0(B_{2r})}\right)
\end{align}
Then, using Corollary \ref{cor2} (with $\beta=\alpha$ by making $\a$ smaller if necessary) and \eqref{eq8}, 
\begin{align*}
\ell^{2\a}(r)[u]_{\mathcal L^{2\a}(B_{cr})} 
&\leq C(\|u\|_{L^\8(B_{1+r})}+\ell^\a(r)[u]_{\mathcal L^\a(B_{r})}+\ell^{1+\a}(r)[L_Ku]_{\mathcal L^\a(B_{r})})\\
&\leq C(\|u\|_{L^\8(B_{1+r})}+\|L_Ku\|^{(1)}_{\mathcal L^0(B_{r})}+\ell^{1+\a}(r)[L_Ku]_{\mathcal L^\a(B_{r})}),
\end{align*}
where $c>0$ is sufficiently small. Note that
\begin{align*}
\ell^{1+\a}(r)[L_Ku]_{\mathcal L^\a(B_r)}
&=\ell^{1+\a}(r) \sup_{\substack{x,y\in B_r\\x\neq y}}\frac{|L_Ku(x)-L_Ku(y)|}{\ell^\a(|x-y|)}\\
&=\ell^{1+\a}(r) \sup_{\substack{x,y\in B_r\\x\neq y}}\frac{|L_Ku(x)-L_Ku(y)|}{\ell(|x-y|)}\ell^{1-\alpha}(|x-y|)\\
&\leq\ell^{2}(r) \sup_{\substack{x,y\in B_r\\x\neq y}}\frac{|L_Ku(x)-L_Ku(y)|}{\ell(|x-y|)}
=\ell^{2}(r)[L_Ku]_{\mathcal L^1(B_r)}.
\end{align*}
Then 
\begin{align*}
\ell^{2\a}(r)[u]_{\mathcal L^{2\a}(B_{cr})} 
\leq C(\|u\|_{L^\8(B_{1+r})}+\|L_Ku\|^{(1)}_{\mathcal L^0(B_{r})}+\ell^{2}(r)[L_Ku]_{\mathcal L^1(B_r)}).
\end{align*}
Now we can repeat this procedure using Corollary \ref{cor2} with $\beta=2\alpha$, and we obtain that 
\begin{align*}
\ell^{3\a}(r)[u]_{\mathcal L^{3\a}(B_{cr})} 
\leq C(\|u\|_{L^\8(B_{1+r})}+\|L_Ku\|^{(1)}_{\mathcal L^0(B_{r})}+\ell^{2}(r)[L_Ku]_{\mathcal L^1(B_r)})
\end{align*}
(with a smaller value for $c>0$).  Repeating this procedure $\lfloor 1/\a \rfloor$ times, we obtain the claim.
\end{proof}

\begin{proof}[Proof of Theorem \ref{thm2}]
The claim follows from Corollary \ref{cor:high} with a covering argument using Lemma~\ref{prop2}.   We give the details for completeness.  Let $\mu>1$, $r_0>0$, $x\in \Omega$, and $r\in(0,r_0)$ such that $B_{\mu r}(x)\subset \Omega$. In the following $C$ denotes possibly different constants independent of $u$.

By Corollary \ref{cor:high}, there is $c\in(0,1)$ such that
\[
\ell^{1+\a}(r)[u]_{\mathcal L^{1+\a}(B_{cr}(x))} \leq C(\|u\|_{L^\8(B_1(\Omega))}+\ell(r)\|L_Ku\|_{L^\8(B_{r}(x))}+\ell^{2}(r)[L_Ku]_{\mathcal L^1(B_r(x))}).
\]
By Lemma \ref{prop2} (adjusting $\mu$ and $r_0$), 
\begin{align*}
\|u\|^{(0)}_{\mathcal L^{1+\a}(\W)} 
&\leq C\sup_{\substack{B_{\m r}(x)\ss\W\\r\in(0,r_0)}} \1\|u\|_{L^\8(B_{cr}(x))} + \ell^{1+\a}(r)[u]_{\mathcal L^{1+\a}(B_{cr}(x))}\2\\
&\leq C\|u\|_{L^\8(B_{1}(\W))} +C\sup_{\substack{B_{\m r}(x)\ss\W\\r\in(0,r_0)}}\1\ell(r)\|L_Ku\|_{L^\8(B_{r}(x))}+\ell^{2}(r)[L_Ku]_{\mathcal L^1(B_r(x))}\2.
\end{align*}
The claim now follows by a final application of Lemma \ref{prop2} to $L_Ku$.
\end{proof}

\subsection{Boundary regularity for zero boundary data}\label{brwzbd}

In this section the goal is to establish the following boundary regularity estimate with zero complementary data.

\begin{theorem}\label{thm:bdry}
Let $\W\ss\R^N$ be a bounded open set with a uniform exterior sphere condition, $L_K$ be a uniformly elliptic operator in $\W$ with parameters $0<\l\leq\L$, $u \in L^\8(B_1(\W))\cap C(\overline\W) \cap C^{Dini}_{loc}(\W)$ with $u=0$ in $B_1(\W)\sm \W$, and $L_K u \in L^\8(\W)$. Then there are $\a=\a(\l,\L,N,\Omega)\in(0,1)$ and $C=C(N,\l,\L,\Omega)>0$ such that $u\in \cL^\a(B_1(\Omega))$ and 
\begin{align}\label{umc}
[u]_{\cL^\a(B_1(\Omega))}\leq C \|L_Ku\|_{L^\8(\W)}.
\end{align}
\end{theorem}


Our proof relies on the method of barriers. The following construction extends the result of \cite[Lemma 5.3]{MR3995092} to the case $K\not\equiv 1$.

\begin{lemma}\label{sub}
Let $L_K$ satisfy the uniform ellipticity condition \ref{itm:uellip} with respect to $0<\l\leq\L$. There are $\a,r_0,\d\in(0,1)$ such that for $r\in(0,r_0]$ and $\varphi(x) := \ell^\a((|x|-r)_+)$, it holds that 
\begin{align*}
L_K\varphi(x) \geq \d \qquad \text{ for any $x \in B_{r+r^2}\sm B_{r}$.}
\end{align*}
\end{lemma}

\begin{proof}
The exponent $\a\in(0,1)$ will be conveniently fixed at the end of the proof. Let $x \in B_{r+r^2}\sm B_{r}$, $d:=|x|-r$, and assume that $r_0$ is sufficiently small such that $B_{r+d} \ss B_1(x)$. Then
\begin{align*}\label{eq3}
L_K\varphi(x) = \1\int_{B_{r+d}} + \int_{B_1(x) \sm B_{r+d}}\2\frac{\varphi(x)-\varphi(y)}{|y-x|^N}K(x,y-x)dy.
\end{align*}

For the first term the integrand is positive by monotonicty (recall that $x\in \partial B_{r+d}$) so that
\[
\int_{B_{r+d}} \frac{\varphi(x)-\varphi(y)}{|y-x|^N}K(x,y-x)dy \geq \l\ell^\a(d)\int_{B_{r}}\frac{dy}{|y-x|^N} \geq c\l\ell^{\a-1}(d),
\]
where we used that $\varphi(y)=0$ for $y\in B_r$ and the last inequality follows by Lemma 5.2 in \cite{MR3995092} (we also included a different proof in Lemma \ref{lem:tobias} below).

For the second term the integrand is negative instead and we now split the integration using
\begin{align*}
&\int_{B_1(x) \sm B_{r+d}}\frac{\varphi(x)-\varphi(y)}{|y-x|^N}K(x,y-x)dy
\geq -\L\int_{B_1(x) \sm B_{r+d}}\frac{\ell^\a(|y|-r)-\ell^\a(d)}{|y-x|^N}dy,\\
&\quad= -\L\1\int_{B_d(x) \sm B_{r+d}} + \int_{(B_1(x) \sm B_{r+d})\sm B_d(x)}\2\frac{\ell^\a(|y|-r)-\ell^\a(d)}{|y-x|^N}dy.
\end{align*}
For the second term we use the triangle inequality $|y|-r \leq |y-x|+d\leq 2|y-x|$ (because $y\in \R^N\sm B_d(x)$) and get that
\begin{align*}
\int_{(B_1(x) \sm B_{r+d})\sm B_d(x)}\frac{\ell^\a(|y|-r)-\ell^\a(d)}{|y-x|^N}dy 
&\leq \int_{(B_1(x) \sm B_{r+d})\sm B_d(x)}\frac{\ell^\a(2|y-x|)-\ell^\a(d)}{|y-x|^N}dy \\
&\leq \int_{B_1\sm B_d}\frac{\ell^\a(2|y|)-\ell^\a(d)}{|y|^N}dy= \w_N \int_d^1 \frac{\ell^\a(2\r)-\ell^\a(d)}{\r}d\r.
\end{align*}
By truncating the integral at $\r_0/2$ we pick up a constant term and can substitute $\ell(2\r) = |\ln(2\r)|^{-1}$, thus
\begin{align*}
\int_d^1 \frac{\ell^\a(2\r)-\ell^\a(d)}{\r}d\r &\leq C + \int_d^{\r_0/2} \frac{|\ln(2\r)|^{-\a}-|\ln d|^{-\a}}{\r}d\r\\
&\leq C + \frac{1}{1-\a}|\ln(2d)|^{1-\a} - |\ln d|^{1-\a}\\
&\leq C + \frac{1}{1-\a}|\ln 2|^{1-\a}+\frac{\a}{1-\a}|\ln d|^{1-\a}\leq \frac{2\a}{1-\a}|\ln d|^{1-\a}.
\end{align*}
In the second to last step we used the sub-additivity of $z\mapsto z^{1-\a}$ and finally absorbed the constants by assuming that $d<r^2<r_0^2$ is sufficiently small (depending on $\a$ to be fixed).

For the remaining term we use that, by the concavity of $z \mapsto \ell^\a(z)$,
\begin{align*}
    \ell^\a(|y|-r)-\ell^\a(d)\leq \alpha \frac{\ell^{\a+1}(d)}{d}(|y|-|x|)
\end{align*}
and, by the triangle inequality $|y|-|x|\leq |y-x|$,
\[
\int_{B_d(x) \sm B_{r+d}} \frac{\ell^\a(|y|-r)-\ell^\a(d)}{|y-x|^N}dy \leq \a\frac{\ell^{\a+1}(d)}{d}\int_{B_d(x) \sm B_{r+d}} \frac{|y|-|x|}{|y-x|^N}dy \leq \a\ell^{\a+1}(d)\w_N.
\]

In conclusion $L_K\varphi(x) \geq (c\l-C\L\a)\ell^{\a-1}(d) \geq \d\ell^{\a-1}(d) \geq \d$ for $\a \in(0,1)$ conveniently small in order for $c\l$ to absorb the term $C\L\a$.
\end{proof}

\begin{proof}[Proof of Theorem \ref{thm:bdry}]

By the maximum principle given in Lemma \ref{mp} we get that
\[
\|u\|_{L^\8(B_1(\W))} \leq C_0\|L_Ku\|_{L^\8(\W)}.
\]

Let $\alpha',r_0,\delta\in(0,1)$ and $\varphi(x):=\ell^{\alpha'}((|x|-r)_+)$ as in Lemma \ref{sub}. Let $x_0\in\partial \Omega$ and without loss of generality, assume that $\overline{B_{r_0}}\cap \overline{\Omega}=\{x_0\}$. Let $\W':=\Omega\cap B_{r_0^2+r_0}$ and 
\begin{align*}
\psi := \|L_K u\|_{L^\infty(\Omega)}\1\frac{C_0}{\ell^{\alpha'}(r^2)}+\frac{1}{\d}\2\varphi,
\end{align*}
such that
\[
\begin{cases}
L_K u \leq L_K\psi \text{ in } \W',\\
u\leq \psi \text{ on } B_1(\W)\sm \W'.
\end{cases}
\]
By the maximum principle it then follows that $u\leq \psi$ in $B_1(\W)$. In general, once we apply this reasoning to every possible boundary point, we obtain a modulus of continuity at the boundary of the form
\[
|u(x)| \leq C\|L_Ku\|_{L^\8(\W)}\ell^{\alpha'}(d(x)).
\]

Finally, we conclude the proof by using the interior regularity estimate Theorem \ref{thm1} and the Lemma \ref{prop4}.
\end{proof}

\section{Dirichlet boundary value problem}\label{dbvp}

It is convenient at this point to define the Banach spaces and norms that are controlled by the a priori estimates of our solutions with zero boundary data.

\begin{definition}\label{XY}
For $\a\in(0,1)$, define the Banach spaces 
\begin{align*}
 &\mathcal X^\alpha(\Omega) := \{u : \R^N\to\R\ | \ \|u\|_{\mathcal X^\alpha(\Omega)}<\8 \text{ and }u=0\text{ in }\R^N\backslash \W\},\\
 &\mathcal Y(\Omega) := \{f:\W\to\R \ | \ \|f\|_{\mathcal Y(\Omega)}<\8\},
\end{align*}
endowed with the norms
\begin{align*}
\|u\|_{\mathcal X^\alpha(\Omega)} := \|u\|_{\cL^{1+\a}(\W)}^{(0)} + [u]_{\cL^\a(\R^N)},\qquad     
\|f\|_{\mathcal Y} := \|f\|_{\cL^{1}(\W)}^{(1)}+\|f\|_{L^\infty(\W)}.
\end{align*}
\end{definition}

\begin{remark}\label{rema}
If $\W\ss\R^N$ is a bounded open set with a uniform exterior ball condition, $L_K$ is a translation invariant (see \ref{itm:ti}), $1$-regular with constant $\L>0$ (see Definition \ref{def:reg}), uniformly elliptic operator with parameters $0<\l\leq \L$ (see \ref{itm:uellip}), and $u\in \mathcal X^\alpha(\Omega)$, then the combination of the a priori estimates given by Theorem~\ref{thm2} and Theorem~\ref{thm:bdry} together with the non-homogeneous maximum principle (Lemma \ref{mp}) implies that
\[
\|u\|_{\mathcal X^\alpha(\Omega)} \leq C\|L_Ku\|_{\mathcal Y(\Omega)}
\]
for some $\a=\a(\l,\L,N,\Omega) \in (0,1)$ and some $C=C(N,\l,\L,\W)>0$.
\end{remark}

Our first goal is to establish the following result.

\begin{theorem}\label{thm:loglap}
Let $\W\ss\R^N$ be a bounded Lipschitz open set with a uniform exterior ball condition, and $L_K$ be a translation invariant uniformly elliptic operator  with respect to $0<\l\leq \L$ and assume that
\begin{align*}
|D K(y)|\leq \L|y|^{-1}\qquad \text{ for every $y\in\R^N\sm\{0\}$}.
\end{align*}
Then, there is $\alpha=\alpha(N,\l,\L,\Omega)>0$ such that the following holds: Given $f\in \mathcal Y(\Omega)$, there is a unique solution $u \in \mathcal X^\alpha(\Omega)$ of 
\begin{align*}
L_Ku=f \text{ in } \W,\qquad 
u=0 \text{ on } B_1(\W)\sm\W.
\end{align*}
Moreover, $\|u\|_{\mathcal X^\alpha(\Omega)} \leq C\|f\|_{\mathcal Y(\Omega)}$ for some $C=C(N,\l,\L,\Omega)$.
\end{theorem}

In order to take advantage of the already established existence theory in \cite{feulefack2021nonlocal} we need to consider smooth kernels and right-hand sides. This can be achieved by a mollification argument detailed next.

Let $L_K$ be a translation invariant uniformly elliptic operator with parameters $0<\lambda\leq \Lambda$ and let $K_i\in C^\infty_c(\R^N)$ be a mollification of $K$ such that, for all $i\in\N$,
\begin{align}
\frac{\lambda}{2}&\leq K_i(y)\quad \text{ for $y\in B_1$,}\quad 0\leq K_i(y)\leq \Lambda \quad \text{ for $y\in\R^N$,}\label{Kibds}\\
K_i&=0\qquad \text{ in }\R^N\backslash B_{1+\frac{1}{i}}(0),\label{Kisop}\\
K_i&\to K\quad \text{ a.e in $\R^N$ and in $L^1(B_1)$ as $i\to\infty$},\label{Ki}\\
\int_{B_{2}}&\left|D K_i\right|\, dy<C\quad \text{ for all  $i\in \N$ and for some $C>0$}\notag.
\end{align}
In particular, we have that
\begin{align}
    \int_{B_{2}\backslash B_{\frac{9}{10}}}&\left|D \left(\frac{K_i(y)}{|y|^N}\right)\right|\, dy<C\quad \text{ for all  $i\in \N$ and for some $C>0$}\label{Ki2}.
\end{align}

For $\alpha\in(0,1)$ and $u\in {\mathcal X}^\a(\Omega)$, let
\begin{align}\label{Ri}
R_{i}u(x):=\int_{\R^N\backslash B_1}\frac{u(x)-u(x+y)}{|y|^N}K_i(y)\, dy,\qquad x\in\Omega.
\end{align}

\begin{lemma}\label{lem:Ri}
Let $u\in \mathcal X^\alpha(\Omega)$, then, passing to a subsequence,
\begin{align*}
\|R_i u\|_{\mathcal Y(\Omega)}\leq \frac{\|u\|_{\mathcal X^\alpha(\Omega)}}{i}+C \|u\|_{L^\infty(\Omega)}\qquad \text{ for all }i\in\N
\end{align*}
and for some constant $C>0.$
\end{lemma}
\begin{proof}
By \eqref{Kibds} and \eqref{Kisop}, we have that
\begin{align}\label{Ktozero}
\int_{\R^N\backslash B_1} K_i(y)\, dy\leq \Lambda |B_{1+\frac{1}{i}}\backslash B_1|=o(1)\qquad \text{ as $i\to\infty$.}    
\end{align}
 Then,
\begin{align*}
    \|R_i u\|_{L^\infty(\Omega)}
    \leq 2\|u\|_{L^\infty(\Omega)}\int_{\R^N\backslash B_1} K_i(y)\, dy\leq o(1)\|u\|_{\mathcal X^\a(\Omega)}\qquad \text{ as }i\to\infty.
\end{align*}
Moreover, observe that, by \eqref{Ktozero},
\begin{align*}
R_iu(x)=u(x)o(1)-\int_{\R^N}\frac{K_i(y-x)}{|y-x|^N}\chi_{\omega_i(x)}u(y)\, dy,\qquad x\in\Omega,\ \omega_i(x):=B_{1+\frac{1}{i}}(x)\backslash B_1(x).
\end{align*}
as $i\to\infty$.  Therefore, for $x,z\in \Omega$,
\begin{align}\label{cl}
    &\frac{|R_i u(x)-R_i u(z)|}{\ell(|x-z|)}
    \leq o(1)\|u\|_{{\mathcal X}^\alpha}
    +\frac{\|u\|_{L^\infty(\Omega)}}{\ell(|x-z|)}\int_{\R^N}\left|\frac{K_i(y-x)}{|y-x|^N}\chi_{\omega_i(x)}-\frac{K_i(y-z)}{|y-z|^N}\chi_{\omega_i(z)}\right|\, dy
\end{align}
as $i\to\infty$. Finally, we claim that \eqref{Ki2} implies that 
\begin{align}\label{annest}
   \frac{1}{\ell(|x-z|)}\int_{\R^N}\left|\frac{K_i(y-x)}{|y-x|^N}\chi_{\omega_i(x)}-\frac{K_i(y-z)}{|y-z|^N}\chi_{\omega_i(z)}\right|\, dy<C
\end{align}
for all $i\in \N$ and $x,z\in \Omega.$  To see this, it suffices to consider $x,z\in \Omega$ such that $|x-z|<\frac{1}{10}$. Let $h_i:B_2\backslash \{0\}\to \R$ be given by $h_i(\zeta):=\frac{K_i(\zeta)}{|\zeta|^N}$. Then, for all $i\in\N$,
\begin{align*}
    &\frac{1}{\ell(|x-z|)}\int_{\omega_i(x)\cap \omega_i(z)}\left|\frac{K_i(y-x)}{|y-x|^N}\chi_{\omega_i(x)}(y)-\frac{K_i(y-z)}{|y-z|^N}\chi_{\omega_i(z)}(y)\right|\, dy\\
    &=\frac{1}{\ell(|x-z|)}\int_{\omega_i(x)\cap \omega_i(z)}\left|h_i(y-x)-h_i(y-z)\right|\, dy\\
    &\leq\frac{|x-z|}{\ell(|x-z|)}\int_{\omega_i(x)\cap \omega_i(z)}\int_0^1\left|D h_i(s(y-x)+(1-s)(y-z))\right|\, ds\, dy\\
    &\leq\frac{|x-z|}{\ell(|x-z|)}\int_{B_2\backslash B_{\frac{9}{10}}}\left|D h_i(y)\right|\, dy
     <C'
\end{align*}
for some $C'>0$; where we used \eqref{Ki2} and the fact that $s(y-x)+(1-s)(y-z)\in B_2\backslash B_{\frac{9}{10}}$ for all $s\in(0,1)$ and $y\in \omega_i(x)\cap \omega_i(z)$.  On the other hand,
\begin{align*}
    &\frac{1}{\ell(|x-z|)}\int_{\omega_i(x)\backslash \omega_i(z)}\left|\frac{K_i(y-x)}{|y-x|^N}\chi_{\omega_i(x)}(y)-\frac{K_i(y-z)}{|y-z|^N}\chi_{\omega_i(z)}(y)\right|\, dy\\
    &\leq \frac{\Lambda}{\ell(|x-z|)}|\omega_i(x)\backslash \omega_i(z)|<C''
\end{align*}
for some $C''>0$.  The same can be argued for the integral over $\{\omega_i(z)\backslash \omega_i(x)\}$, and \eqref{annest} follows. This ends the proof.
\end{proof}
 
 Recall that $d(x):=\dist(x,\partial \Omega)$ and $d(x,y):=\min(d(x),d(y))$.
 \begin{lemma}\label{lem:conv}
 Let $u \in C(\overline\W) \cap \cL^{1+\a}(\W)$ and $(u_{i}) \subset \cL^{1+\a/2}_{loc}(\W)\cap C(\overline \W)$ be such that $u_{i}\to u$ uniformly in $\overline\W$ and
 \begin{align}\label{l}
\lim_{i\to\infty}\|u_i-u\|_{\cL^{1+\a/2}(\Omega)}^{(0)}=\sup_{\substack{x,z\in \Omega\\x\neq y}}\ell^{1+\a/2}(d(x,z))\frac{|u_i(x)-u(x)-u_i(z)+u(z)|}{\ell^{1+\a/2}(|x-z|)}=0.     
 \end{align}
 Then,
\begin{align*}
L_{K_i}u_{i}\to L_{K}u\qquad \text{ locally uniformly in $\Omega$ as $i\to\infty$}.
\end{align*}
 \end{lemma}
 \begin{proof}
 Let $U\subset\subset \Omega$. By dominated convergence,
 \begin{align*}
H_i(x):=\int_{B_1}\left((u(x)-u(x+y))(K_i(y)-K(y))\right)|y|^{-N}\, dy=o(1)\qquad \text{ as }i\to\infty \text{ uniformly in $U$.}
 \end{align*}
Let $\delta:=\dist(U,\partial \Omega)/2.$ Then, for $x\in U,$
\begin{align*}
I_i(x)
&:=\int_{B_\delta}\left((u_i(x)-u(x)-u_i(x+y)+u(x+y))K_i(y)\right)|y|^{-N}\, dy\\
&\leq \Lambda\|u_i-u\|_{\cL^{1+\a/2}(\Omega)}^{(0)}\int_{B_\delta}\ell^{-1-\a/2}(d(x,x+y))\frac{\ell^{1+\a/2}(|y|)}{|y|^{N}} \, dy.
\end{align*}
Note that this last integral is uniformly bounded for all $x\in U$. Moreover, 
\begin{align*}
J_i(x)
&:=\int_{B_1\backslash B_\delta}\left((u_i(x)-u(x)-u_i(x+y)+u(x+y))K_i(y)\right)|y|^{-N}\, dy\\
&\leq 2\Lambda\delta^{-N}\|u_i-u\|_{L^\infty(\Omega)}|B_1\backslash B_\delta|.
\end{align*}
Therefore,
\begin{align*}
L_{K_i}u_{i}(x)-L_{K}u(x)
&=\int_{B_1}\left((u_i(x)-u_i(x+y))K_i(y)-(u(x)-u(x+y))K(y)\right)|y|^{-N}\, dy\\
&=H_i(x)+I_i(x)+J_i(x)=o(1)
 \end{align*}
uniformly in $U$ as $i\to\infty$. Since $U$ is an arbitrary compact subset of $\Omega$, the claim follows.
 \end{proof}
 
\begin{proof}[Proof of Theorem \ref{thm:loglap}]
Let $(f_i) \subset C^\8(\W)$ be a sequence that converges locally uniformly to $f$ and such that $\sup_i\|f_i\|_{\mathcal Y(\Omega)} <\8$. This approximation can be achieved for instance by a mollification of $f$ and keeping in mind that $\|f\|_{\mathcal Y(\Omega)}<\8$ (see Lemma \ref{lem:10}).  Let $K_i\in C^\infty_c(\R^N)$ be a mollification of $K$ as in \eqref{Ki}. Then, by the existence result from \cite[Theorem 1.6]{feulefack2021nonlocal} in the smooth setting ($m=\8$), there is  $u_{i}\in L^\8(B_1(\W))\cap C^\8_{loc}(\W)$ such that 
\begin{align*}
\int_{\R^N}\frac{u_{i}(x)-u_{i}(y)}{|x-y|^N}K_i(x-y)\, dy = f_i(x)\qquad \text{ for }x\in\Omega,\qquad u_{i}=0\quad \text{ in }\R^N\backslash \Omega
\end{align*}
for every $i\in \N$.  Then,
\begin{align*}
L_{K_i}u_{i}(x)=\int_{B_1}\frac{u_{i}(x)-u_{i}(x+y)}{|y|^N}K_i(y)\, dy = f_i(x) - R_i u_{i}(x)\qquad \text{ for }x\in\Omega,
\end{align*}
where $R_i$ is given by \eqref{Ri}.

Arguing as in \cite[Theorem 5.4]{MR3995092} using a comparison principle for weak solutions (see \cite[Proposition 1.1]{feulefack2021nonlocal}) and the subsolution given in Lemma \ref{sub}, it follows that $u_{i} \in C(\overline\W)$ (we omit the details since this argument is basically the same as the one given in \cite[Theorem 5.4]{MR3995092} with simple changes). Hence, we have produced a sequence of classical solutions in the regularity classes of our a priori estimates. Then,
\begin{align}\label{pe}
\|u_{i}\|_{\mathcal X^\alpha(\Omega)} \leq C(\|f_{i}\|_{\mathcal Y(\Omega)}+\|R_{i}u_{i}\|_{\mathcal Y(\Omega)})\qquad \text{ for all }i\in \N
\end{align}
(see Remark \ref{rema}).  

By Lemma \ref{lem:Ri} and Remark \ref{othermp}, passing to a subsequence,
\begin{align}\label{pe2}
\|R_{i}u_{i}\|_{\mathcal Y(\Omega)}
\leq \frac{1}{i}\|u_{i}\|_{\mathcal X^\alpha(\Omega)}+C\|u\|_{L^\infty(\Omega)}
\leq \frac{1}{i}\|u_{i}\|_{\mathcal X^\alpha(\Omega)}+C'\|f\|_{L^\infty(\Omega)}
\qquad \text{ as }i\to\infty
\end{align}
for some $C'>0$. Then, \eqref{pe} and \eqref{pe2} imply that $(u_{i}) \subset \cL^{1+\a/2}_{loc}(\W)\cap C(\overline \W)$ is a compact sequence and we can extract a sub-sequence (denoted also by $u_{i}$) such that, for some $u \in C(\overline\W) \cap \cL^{1+\a}(\W)$,  $u_{i}\to u$ uniformly in $\overline\W$ and \eqref{l} holds, see Lemma \ref{cpt}. 

Then, by Lemma \ref{lem:conv}, $L_{K_i}u_{i}$ and $f_i-R_iu_{i}$ converges locally uniformly to $L_{K}u$ and $f$ respectively; therefore, $L_Ku=f$ in $\Omega$ as claimed. Since uniqueness follows from the maximum principle Lemma~\ref{lem:mp}, this ends the proof.
\end{proof}

We are ready to show Theorem \ref{fredholm} stated in the introduction. 

\begin{proof}[Proof of Theorem \ref{fredholm}]
From the previous theorem we know that $L_K^{-1}:\mathcal Y(\Omega)\to\mathcal X^\alpha(\Omega)$ is a bounded linear operator. 

Next, we show that $u\in\mathcal X^\a(\Omega)$ is a solution of \eqref{fa} if and only if $u$ solves
\begin{align}\label{fa2}
(I+L_K^{-1}T)u=L_K^{-1}f.
\end{align}
If $u$ is a solution of \eqref{fa}, then, by applying $L_K^{-1}$ on both sides  of \eqref{fa}, we obtain that $u$ solves \eqref{fa2}. On the other hand, if $u$ solves \eqref{fa2}, then $u = L_K^{-1}(f-Tu)$ and therefore $(L_K+T)u=f$.

Since $L_K$ is bounded and $T$ is compact, we have that the composition $L_K^{-1}T$ is a compact operator. Then, the Fredholm Alternative (see \cite[Sec. 21.1, Theorem 5]{lax2002functional}) yields that either there exists a unique solution for each $f\in \mathcal Y(\Omega)$ or there exists a non-trivial $u\in \ker(I+L_K^{-1}T)$. Reasoning as before, this is equivalent to the existence of a non-trivial $u \in \mathcal X^\alpha(\Omega)$ such that $(L_K+T)u=0$.

The estimate for the first alternative is a consequence of the inverse mapping theorem.
\end{proof}

\subsection{Applications}

\subsubsection{The logarithmic Laplacian}

Following \cite{MR3995092}, the logarithmic Laplacian is  given by
\begin{align}\label{lld}
L_\D u = c_NLu - c_N J\ast u +\r_Nu,    
\end{align}
where $L=L_K$ with $K\equiv 1$, $J(x):= \chi_{\R^N\sm B_1}|x|^{-N}$, 
\begin{align}\label{constants}
c_N:=\pi^{-\frac{N}{2}}\Gamma(\tfrac{N}{2}),\qquad \rho_N:=2\ln 2 + \psi(\tfrac{N}{2})-\gamma,
\end{align}
$\gamma=-\Gamma'(1)$ is the Euler-Mascheroni constant, and $\psi=\frac{\Gamma'}{\Gamma}$ is the Digamma function.

We show first some auxiliary lemmas.

\begin{lemma}\label{bd:lem}
Let $R>0$ and, for $w\in B_R(0)$, let
\begin{align*}
\Psi(w):=\frac{1}{\ell(|w|)}\int_{\Omega}\left|\chi_{\R^N\backslash B_1}(z)|z|^{-N}-\chi_{\R^N\backslash B_1}(w+z)|w+z|^{-N}\right|\, dz<\infty.
\end{align*}
Then $\sup_{B_R}\Psi<\infty$.
\end{lemma}
\begin{proof}
Note that
\begin{align*}
M:=\sup_{w\in B_R(0)}\frac{1}{\ell(|w|)}\int_{\Omega}\left|\chi_{B_1}(z)-\chi_{B_1}(w+z)\right|\, dz
=\sup_{w\in B_R(0)}\frac{2|B_1\backslash B_1(w)|}{\ell(|w|)}
<\infty.
\end{align*}
Moreover, since the function $f:\R^N\to\R$ given by $f(z)=\chi_{\R^N\backslash B_1}(z)|z|^{-N}+\chi_{B_1}(z)$ is Lipschitz continuous, we have that
\begin{align*}
\sup_{w\in B_R(0)}\Psi(w)\leq \sup_{w\in B_R(0)}\frac{1}{\ell(|w|)}\int_{\Omega}\left|f(z)-f(w+z)\right|\, dz+M<\infty.
\end{align*}
\end{proof}

Recall the spaces $\mathcal X^\alpha(\Omega)$ and $\mathcal Y(\Omega)$ given in Definition \ref{XY}.
\begin{lemma}\label{T:comp}
The operator $T:\mathcal X^\alpha(\Omega)\to\mathcal Y(\Omega)$ given by $T u:=- c_N J\ast u +\r_Nu$ is compact.
\end{lemma}
\begin{proof}
Indeed, $\r_N u$ is compact because the identity $id:\mathcal X^\alpha(\Omega)\to \mathcal Y(\Omega)$ is compact, by Lemma \ref{cpt:local}.  In particular, if $(u_n)$ is a uniformly bounded sequence in $\mathcal X^\alpha(\Omega)$, then, passing to a subsequence, there is $u\in \mathcal X^\alpha(\Omega)$ such that $u_n\to u$ in $\mathcal Y(\Omega)$; thus,
\begin{align}\label{Linfty}
\lim_{n\to\infty}\|u_n-u\|_{L^\infty(\Omega)}=0.
\end{align}
Then, if $R>0$ is such that $\Omega\subset B_R(0)$, 
\begin{align*}
\|J\ast u_n-J\ast u\|_{L^\infty(\Omega)}
&\leq \sup_{x\in \Omega}\int_{B_{2R}(0)\backslash B_1}\frac{|u_n(x-y)-u(x-y)|}{|y|^N}\, dy\\
&\leq \|u_n-u\|_{L^\infty(\Omega)}|B_{2R}\backslash B_1|=o(1)
\end{align*}
as $n\to\infty$.  On the other hand,
\begin{align*}
    &\|J\ast u_n-J\ast u\|^{(1)}_{\cL^1(\Omega)}\leq \ell^{2}(\diam(\Omega))\sup_{\substack{x,y \in \W\\x\neq y}}\ell^{2}(d(x,y))\frac{|J\ast u_n(x)-J\ast u(x)-J\ast u_n(y)+J\ast u(y)|}{\ell(|x-y|)}\\
    &\leq \ell^{2}(\diam(\Omega))\sup_{\substack{x,y \in \W\\x\neq y}}\int_{\Omega}|u_n(z)-u(z)|\frac{|\chi_{\R^N\backslash B_1}(x+z)|x+z|^{-N}-\chi_{\R^N\backslash B_1}(y+z)|y+z|^{-N}|}{\ell(|x-y|)}\, dz\\
    &\leq \|u_n-u\|_{L^\infty(\Omega)}\ell^{2}(\diam(\Omega))\sup_{\substack{x,y \in \W\\x\neq y}}\int_{\Omega}\frac{\left|\chi_{\R^N\backslash B_1}(z)|z|^{-N}-\chi_{\R^N\backslash B_1}(x-y+z)|x-y+z|^{-N}\right|}{\ell(|x-y|)}\, dz\\&
    =o(1)
\end{align*}
as $n\to\infty,$ by Lemma \ref{bd:lem} and \eqref{Linfty}.  As a consequence, the map $u\mapsto- c_N J\ast u$ is compact from $\mathcal X^\alpha(\Omega)$ to $\mathcal Y(\Omega)$.  
\end{proof}

\begin{proof}[Proof of Theorem \ref{loglap:cor}]
The claim follows from Corollary \ref{fredholm} and Lemma \ref{T:comp}.
\end{proof}

We also include in this section a lower-order regularity result for the logarithmic Laplacian.  The following a priori estimate does not require the Fredholm alternative and it is essentially a corollary of Theorem \ref{thm:bdry}. 
\begin{corollary}
Let $\W\ss\R^N$ be a bounded open set with a uniform exterior sphere condition, $f\in L^\infty(\Omega)$, and let $u \in L^\8(B_1(\W))\cap C(\R^N) \cap C^{Dini}_{loc}(\W)$ be a solution of 
\begin{align*}
L_\Delta u=f\quad \text{ in }\Omega, \qquad u=0\quad \text{ in }\R^N\sm \W.
\end{align*}
Then there are $\a=\a(\Omega)\in(0,1)$ and $C=C(\Omega)>0$ such that
\begin{align*}
 \sup_{\substack{x,y\in \R^N\\x\neq y}}\frac{|u(x)-u(y)|}{\ell^\a(|x-y|)}\leq C (\|f\|_{L^\8(\W)}+\|u\|_{L^\8(\W)}).
\end{align*}
\end{corollary}
\begin{proof}
Recall that, by \eqref{lld}, $L_\D u = c_NLu - c_N J\ast u +\r_Nu$. For $x\in \Omega$,
\begin{align*}
|J\ast u(x)|
\leq \int_{\Omega\backslash B_1(x)}\frac{|u(y)|}{|x-y|^{N}}\, dy
\leq |\Omega| \|u\|_{L^\8(\W)},
\end{align*}
and then, $\|L u\|_{L^\8(\W)}\leq C_1 (\|L_\Delta u\|_{L^\8(\W)}+\|u\|_{L^\8(\W)})=C_1 (\|f\|_{L^\8(\W)}+\|u\|_{L^\8(\W)}))$ for some $C_1=C_1(\Omega)>0$. The result now follows from Theorem \ref{thm:bdry}.
\end{proof}

\subsubsection{The logarithmic Schrödinger equation}

\begin{proof}[Proof of Theorem \ref{Sch:log}]
Arguing as in Lemma \ref{T:comp} it follows that $T=(I-\D)^{\log}-c_NL_K$ with $K\equiv 1$ is a compact operator from $\mathcal X^\alpha(\Omega)$ to $\mathcal Y(\Omega)$.   Moreover, in \cite[Section 2]{feulefack2021logarithmic} it is proved that the trivial solution is the only (weak) solution of the homogeneous equation with zero boundary data. The claim now follows from Corollary \ref{fredholm}.
\end{proof}

\section{Open problems}

We close this paper with some open questions that we believe are interesting and would improve  our understanding of the fine regularity properties of this kind of operators. 

\begin{enumerate}
    \item What is the optimal higher-order interior regularity estimate?  For instance, is it true that for $\a\in(0,1)$ one can obtain a bound in the $\cL^{1+\a}$ norm of the solution in terms of the $\cL^\a$ norm of the right-hand side?   This would allow the implementation of a bootstrap argument for bounded solutions in nonlinear problems. 
    \item In bounded domains, Lemma  \ref{mp} implies that solutions are bounded if $L_Ku$ is bounded.  Is it possible to allow a behavior like $L_Ku \sim \ell^{-\a}(d)$ for some small $\a>0$ to obtain an $L^\infty$ estimate on the solution?
    \item What is the exact rate of growth for the constant in the maximum principle? (see Remark \ref{rmk:mp}).
\end{enumerate}

\section{Appendix}

\subsection{Properties for the log-Hölder moduli of continuity}\label{sec:log:holder}

\begin{proof}[Proof of Lemma \ref{prop1} (Semi-homogeneity)]

We divide the estimate according to the regions shown in Figure \ref{fig:1}. Keep in mind that $\r_0=0.1$ and $\ell(\r) = |\ln(\min(\r,\r_0))|^{-1}.$  We also use that, for any $\l>0$,
\[
\lim_{r\to 0^+} \frac{\ell(\l r)}{\ell(r)} = 1.
\]
\begin{figure}
    \centering
    \includegraphics[scale=.6]{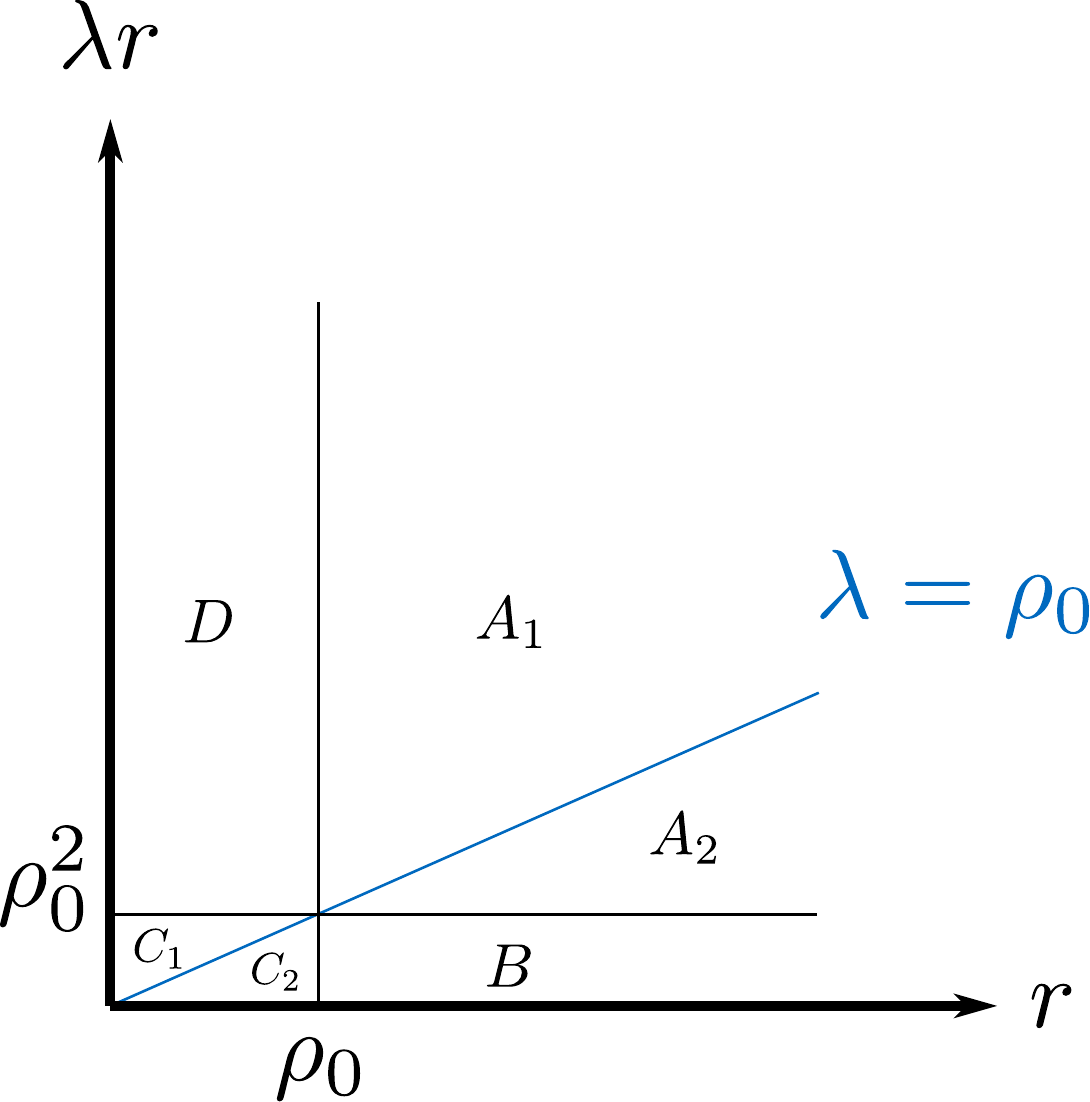}
    \caption{Different configurations for $r$ and $\l r$.}
    \label{fig:1}
\end{figure}

Let $(r,\l r)\in A_1\cup A_2\cup D$. Since $\l r \geq \r_0^2$ we have that $\ell(\l r)\geq \ell(\r_0^2):= c \geq c\ell(\l)\ell(r)$.

If $(r,\l r)\in B$, then
\[
\frac{\ell(\l r)}{\ell(r)} = \ell(\l r) \geq \ell(\l \r_0)
\geq  c \ell(\l), \qquad c:=\inf_{\l\in (0,\r_0)}\ell(\l\r_0)/\ell(\l)>0.
\]

For $(r,\l r)\in C_1$,
\[
\ell(\l) = 1 = C\inf_{r\in (0,\r_0)}\frac{\ell(\r_0 r)}{\ell(r)} \leq C\frac{\ell(\l r)}{\ell(r)}.
\]

Finally, if $(r,\l r)\in C_2$ we use the monotonicity of $z\mapsto z/(z+|\ln \l|)$ for $z>0$ and the fact that $|\ln z|>1$ for $z\in(0,\r_0)$,
\[
\frac{\ell(\l r)}{\ell(r)} = \frac{|\ln r|}{|\ln (\l r)|} = \frac{|\ln r|}{|\ln r|+|\ln \l|} \geq \frac{1}{1+|\ln \l|} \geq
\frac{1}{2|\ln\l|}=\frac{1}{2}\ell(\l).
\]
\end{proof}

\begin{lemma}[Completeness]\label{banach}
Let $\a\geq0$ and $\W\ss\R^N$ arbitrary. Then $(\cL^\a(\W),\|\cdot\|_{\cL^\a(\W)})$ is a Banach space.
\end{lemma}

\begin{proof}
Let $(u_i) \ss \cL^\a(\W)$ a Cauchy sequence. We have that for every $x\in\W$ the sequence $(u_i(x)) \ss\R$ is also Cauchy and hence converges to some $u(x)$. We need to show that $u\in \cL^\a(\W)$ and $\|u_i-u\|_{\cL^\a(\W)}\to 0$.

To show that $u\in \cL^\a(\W)$ we use that any Cauchy sequence is bounded in norm, therefore for $x,y\in\W$ different we get that
\[
\frac{|u(x)-u(y)|}{\ell^\a(|x-y|)} = \lim_{j\to\8}\frac{|u_j(x)-u_j(y)|}{\ell^\a(|x-y|)} \leq \limsup_{j\to\8}[u_j]_{\cL^\a(\W)}
\]
and the right-hand side is bounded independently of $x$ and $y$.

Let $\e>0$ be fixed and consider then $k\in\N$ such that for $i,j\geq k$
\[
\|u_i-u_j\|_{\cL^\a(\W)}  = \|u_i-u_j\|_{L^\8(\W)}+\sup_{\substack{x,y\in\W\\ x\neq y}} \frac{|(u_i-u_j)(x)-(u_i-u_j)(y)|}{\ell^\a(|x-y|)} \leq \e.
\]

For $i\geq k$ and any $x\in\W$ we get
\[
|(u_i-u)(x)| = \lim_{j\to\8}|(u_i-u_j)(x)| \leq \e.
\]
This means that $\|u_i-u\|_{L^\8(\W)}\leq \e$.

Also for $i\geq k$ and $x,y\in\W$
\[
|(u_i-u)(x)-(u_i-u)(y)| = \lim_{j\to\8}|(u_i-u_j)(x)-(u_i-u_j)(y)| \leq \e\ell^\a(|x-y|).
\]
After dividing by $\ell^\a(|x-y|)$ and taking the supremum over $x,y\in\W$ (different) we get that $[u_i-u]_{\cL^\a(\W)}\leq \e$ which completes the norm convergence of $u_i\to u$.
\end{proof}

For the following proofs recall that for $x,y \in \W\ss\R^N$, $d(x):=\dist(x,\p\W)$ and $d(x,y):=\min(d(x),d(y))$.

\begin{proof}[Proof of Lemma \ref{prop2}]
In the following, we use $C>0$ to denote possibly different constants. Let $\a,\b\geq 0$, $r_0>0$, $\l>1$, $R:=\sup\limits_{x\in\W} d(x)$, $A := \|u\|^{(\b)}_{\mathcal L^\a(\W)} = A_1+A_2,$ and $B := B_1+B_2$, where
\begin{align*}
A_1&:=\sup_{x\in\W} \ell^{\b}(d(x))|u(x)|,\quad 
A_2:=\sup_{\substack{x,y \in \W\\x\neq y}} \ell^{\a+\b}(d(x,y))\frac{|u(x)-u(y)|}{\ell^\a(|x-y|)},
\\
B_1&:=\sup_{\substack{B_{\l r}(x)\ss\W\\r\in(0,r_0]}} \ell^{\b}(r)\|u\|_{L^\8(B_{r}(x))},\quad 
B_2:=\sup_{\substack{B_{\l r}(x)\ss\W\\r\in(0,r_0]}}\ell^{\a+\b}(r)[u]_{\mathcal L^{\a}(B_{r}(x))}.
\end{align*}

Let $x_0 \in\W$, $\delta=d(x_0)\leq R$, and $r:=\min(\frac{\delta}{\l},r_0)$. Then $B_{\l r}(x_0) \ss \W$ and $r\in(0,r_0]$. By Lemma~\ref{prop1},
\[
\ell^\b(\delta)|u(x_0)| \leq C \frac{\ell^\b(r)}{\ell^\b(\frac{r}{\delta})} \|u\|_{L^\8(B_r(x_0))} \leq \frac{C}{\ell^\b(\min(\frac{1}{\l},\frac{r_0}{R}))} B_1.
\]
Taking the supremum with respect to $x_0 \in\W$, we get that $A_1\leq C B_1$.

On the other hand, let $B_{\l r}(x_0) \ss \W$ with $r\in(0,r_0)$. Then, 
\[
\ell^\b(r)\|u\|_{L^\8(B_r(x_0))} \leq \frac{C}{\ell^\b(\l-1)}\sup_{x \in B_r(x_0)} \ell^\b((\l-1)r)|u(x)|\leq \frac{C}{\ell^\b(\l-1)}A_1,
\]
where we used that $d(x)\geq d(x_0)-|x-x_0|\geq \lambda r - r=(\lambda-1)r$. Taking the supremum with respect to $B_{\l r}(x_0) \ss \W$ with $r\in(0,r_0)$ we obtain that $B_1\leq C A_1$.

Next, let $x_0,y_0 \in\W$ be two different points such that $\delta := d(x_0)\leq d(y_0)$ and let  $r:=\min(\frac{\delta}{\l},r_0)$. If $y_0\in B_r(x_0)$, then
\[
\ell^{\a+\b}(\delta)\frac{|u(x_0)-u(y_0)|}{\ell^\a(|x_0-y_0|)}
\leq C\frac{\ell^{\a+\b}(r)}{\ell^{\a+\b}(\frac{r}{\delta})}\frac{|u(x_0)-u(y_0)|}{\ell^\a(|x_0-y_0|)} 
\leq \frac{C}{\ell^{\a+\b}(\min(\frac{1}{\l},\frac{r_0}{R}))} B_2.
\]
If otherwise, $|x_0-y_0|\geq r$, then $\ell^\alpha(|x_0-y_0|)\geq \ell^\alpha(r)$ and therefore
\[
\ell^{\a+\b}(\delta)\frac{|u(x_0)-u(y_0)|}{\ell^\a(|x_0-y_0|)}
\leq C\ell^\a(r)\frac{\ell^{\b}(r)}{\ell^{\b}(\frac{r}{\delta})}\frac{|u(x_0)|+|u(y_0)|}{\ell^\a(r)} 
\leq \frac{C}{\ell^\b(\min(\frac{1}{\l},\frac{r_0}{R}))}B_1.
\]
Then, taking the supremum with respect to $x_0,y_0 \in\W$, we obtain that $A_2\leq C (B_1+B_2)=CB$. 

Finally, let $B_{\l r}(x_0) \ss \W$ with $r\in(0,r_0)$. Then, 
\[
\ell^{\a+\b}(r)[u]_{\mathcal L^\a(B_{r}(x_0))} \leq \frac{C}{\ell^{\a+\b}(\l-1)}\sup_{\substack{x,y\in B_r(x_0)\\x\neq y}}\ell^{\a+\b}((\l-1)r)\frac{|u(x)-u(y)|}{\ell^\a(|x-y|)} \leq \frac{C}{\ell^{\a+\b}(\l-1)}A_2.
\]
Taking the supremum with respect to $B_{\l r}(x_0) \ss \W$ with $r\in(0,r_0)$, we have that $B_2\leq CA_2$. This concludes the proof. 
\end{proof}

\begin{proof}[Proof of Lemma \ref{prop4}]
Recall that $d(x) := \dist(x,\p\W)$ and $d(x,y) := \min(d(x),d(y))$. Let
\[
C_0:= [u]^{(\b)}_{\mathcal L^\a(\W)} + \sup_{\substack{x\in\p\W\\y\in\overline\W\sm\{x\}}}\frac{|u(x)-u(y)|}{\w_0(|x-y|)}<\infty,
\]
and
\[
\w(r) := \sup_{\substack{x,y\in \overline{\W}\\|x-y|\leq r}}\min(\ell^{-(\a+\b)}(d(x,y))\ell^\a(r),\w_0(d(x))+\w_0(d(y))+\w_0(d(x)+d(y)+r)).
\]
This definition guarantees that 
\[
|u(x)-u(y)| \leq C_0\w(|x-y|)\qquad \text{for any $x,y\in \overline{\W}$}.
\]
Indeed, the $\cL^\a$-regularity hypothesis means that
\[
|u(x)-u(y)| \leq C_0\ell^{-(\a+\b)}(d(x,y))\ell^\a(|x-y|) \qquad \text{for any $x,y\in \overline{\W}$}.
\]
Meanwhile the $\w_0$-regularity hypothesis together with the triangle inequality means instead that
\begin{align*}
|u(x)-u(y)| &\leq C_0(\w_0(|x-x_0|) +\w_0(|y-y_0|)+\w_0(|x_0-y_0|)),\\
&\leq C_0(\w_0(d(x))+\w_0(d(y))+\w_0(d(x)+d(y)+|x-y|))
\end{align*}
where $x_0,y_0\in \p\W$ are such that $d(x) = |x-x_0|$ and $d(y)=|y-y_0|$.

Let us show now that $\w$ is a modulus of continuity. Note that $\w$ is positive and increasing because each one of the factors in the minimum satisfy these properties. Ir remains show that $\w(r)\to 0$ as $r\to0^+$.

Proceeding by contradiction and using the monotonicity of $\w$, assume that there is $\e_0>0$ such that $\w\geq\e_0$ in $(0,\infty)$. This means that both factors in the minimum are also $\geq \e_0$.

Given $\e\in(0,\e_0)$ let $r=r(\e)>0$ be sufficiently small such that $\ell^\a(r)\leq \e^2$. By looking at the first factor in the minimum, we have that, for any $x,y\in\overline{\W}$ with $|x-y|\leq r$, $\ell^{\a+\b}(d(x,y)) \leq \e$. Hence, as $\e$ approaches zero, we can assume that $d(x)$ and $d(y)$ are arbitrarily small. For any $\e_1>0$ there exists in this way some $\e \in(0,\e_0)$ such that
\[
\w_0(d(x))+\w_0(d(y))+\w_0(d(x)+d(y)+r) \leq \frac{\e_1}{2} + \w_0(\e_1+r).
\]
Finally, let $\e_1$ and $r$ be sufficiently small such that the right-hand side above is less than $\e_0/2$, which yields a contradiction.

Next, we show that we can take $\w \sim \ell^{\gamma}$ if $\w_0 = \ell^{\a'}$. For $\gamma\in(0,\min(\a,\a'))$, to be conveniently fixed later on, our goal is to bound the right-hand side below
\[
\frac{\w(r)}{\ell^{\gamma}(r)} =  \sup_{\substack{x,y\in \overline{\W}\\|x-y|\leq r}} W(x,y,r)
\]
where
\[
W(x,y,r) := \min\left(\ell^{-(\a+\b)}(d(x,y))\ell^{\a-\gamma}(r),\frac{\ell^{\a'}(d(x))+\ell^{\a'}(d(y))+\ell^{\a'}(d(x)+d(y)+r)}{\ell^{\gamma}(r)}\right)
\]

We consider some cases.

\underline{Case 1:} If 
\begin{align*}
    \ell^{-(\a+\b)}(d(x,y))\ell^{\a-\gamma}(r)<1,
\end{align*}
then $W(x,y,r)<1$.

\underline{Case 2:} Assume without loss of generality, $d(x)\leq d(y)$. If
\begin{align*}
1\leq \ell^{-(\a+\b)}(d(x,y))\ell^{\alpha-\gamma}(r)
= \ell^{-(\a+\b)}(d(x))\ell^{\alpha-\gamma}(r)
\end{align*}
Then $\ell^{\a+\b}(d(x))\leq \ell^{\a-\gamma}(r)$.

We now consider two sub-cases.

\underline{Case 2.1:} Suppose that $d(x)\leq r$. Then $\frac{d(x)}{r}+1<2$ and, by Lemma \ref{prop1},
\begin{align*}
\frac{\ell^{\alpha'}(d(y))}{\ell^{\gamma}(r)}
\leq \frac{\ell^{\alpha'}(d(x)+r)}{\ell^{\gamma}(r)}
\leq \frac{\ell^{\a'-\gamma}(d(x)+r)}{c\ell^{\gamma}(\frac{r}{d(x)+r})}
\leq C\ell^{-\gamma}\1\frac{1}{d(x)/r+1}\2,
\end{align*}
which is bounded. Similarly, 
\begin{align*}
\frac{\ell^{\alpha'}(d(y)+d(x)+r)}{\ell^{\gamma}(r)}
\leq \frac{\ell^{\alpha'}(2(d(x)+r))}{\ell^{\gamma}(r)}
\leq C\ell^{-\gamma}\1\frac{1}{2d(x)/r+2}\2
\end{align*}
is also bounded. Therefore, $W$ is bounded in this case.

\underline{Case 2.2:} Suppose that $d(x)>r$. In this last case we will use that
\[
\ell^{\gamma}(r) \geq \ell^{\gamma(\a+\b)/(\a-\gamma)}(d(x)).
\]
Actually it is time to choose $\gamma$ sufficiently small so that the exponent in the right-hand side is smaller than $\a'$ and hence $\ell^{\gamma}(r) \geq \ell^{\a'}(d(x))$.

Then we proceed as before
\begin{align*}
\frac{\ell^{\alpha'}(d(y))}{\ell^{\gamma}(r)}
\leq \frac{\ell^{\alpha'}(d(x)+r)}{\ell^{\gamma}(r)}
\leq \frac{\ell^{\alpha'}(2d(x))}{\ell^{\alpha'}(d(x))}
\leq C.
\end{align*}
Similarly, 
\begin{align*}
\frac{\ell^{\alpha'}(d(y)+d(x)+r)}{\ell^{\gamma}(r)}
\leq \frac{\ell^{\alpha'}(2d(x)+2r)}{\ell^{\gamma}(r)}
\leq \frac{\ell^{\alpha'}(4d(x))}{\ell^{\alpha'}(d(x))}
\leq C.
\end{align*}
Therefore, $W$ is also bounded in this case.

Since these are all the possible cases, we have that $W$ is bounded and the claim follows. 
\end{proof}

\begin{lemma}[Compactness]\label{cpt}
Let $\a_1>\a_0\geq 0$ and $\W\ss\R^N$ bounded. For any sequence $(u_i) \ss \cL^{\a_1}(\W)$ such that $M:= \sup_i\|u_i\|_{\cL^{\a_1}(\W)}<\8$ there exist $u \in \cL^{\a_1}(\W)$, and a sub-sequence $u_{i_j}$ such that $u_{i_j}\to u$ with respect to the $\cL^{\a_0}(\W)$ norm.
\end{lemma}

\begin{proof}
Each $u_i$ can be extended to $\overline\W$ preserving its norm
\[
u_i(x) := \inf\{u_i(y) + \ell^\a(|x-y|) \ | \ y\in\W\}.
\]
We get in this way a uniformly bounded and equicontinuous sequence of functions over the compact set $\overline\W$. By Arzelá-Ascoli, there exist $u \in C(\W)$, and a sub-sequence $u_{i_j}$ such that $u_{i_j}\to u$ uniformly. We will show then that $u\in \cL^{\a_1}(\W)$ and $\|u_{i_j}-u\|_{\cL^{\a_0}(\W)}\to 0$.

To show that $u\in \cL^{\a_1}(\W)$ we use that, for $x,y\in\W$ different,
\[
\frac{|u(x)-u(y)|}{\ell^{\a_1}(|x-y|)} = \lim_{j\to\8}\frac{|u_{i_j}(x)-u_{i_j}(y)|}{\ell^{\a_1}(|x-y|)} \leq \sup_{j}[u_{i_j}]_{\cL^{\a_1}(\W)}.
\]
Since the right-hand side is bounded independently of $x$ and $y$, we have that  $\|u\|_{\cL^{\a_1}(\W)} \leq M$.

Let $\e>0$ be fixed, our goal is now to bound by $\e$
\[
\limsup_{j\to\8}\sup_{\substack{x,y\in\W\\x\neq y}} \frac{|(u_{i_j}-u)(x)-(u_{i_j}-u)(y)|}{\ell^{\a_0}(|x-y|)}.
\]

We consider two scenarios depending on
\[
\d := \sup\{d>0 \ | \ \ell^{\a_1-\a_0}(d) \leq \e/(2 M)\} >0.
\]
If $|x-y|<\d$, we get that
\begin{align*}
\frac{|(u_{i_j}-u)(x)-(u_{i_j}-u)(y)|}{\ell^{\a_0}(|x-y|)} &\leq \frac{|u_{i_j}(x)-u_{i_j}(y)|}{\ell^{\a_0}(|x-y|)}+\frac{|u(x)-u(y)|}{\ell^{\a_0}(|x-y|)}\\
&\leq \ell^{\a_1-\a_0}(|x-y|)\1\frac{|u_{i_j}(x)-u_{i_j}(y)|}{\ell^{\a_1}(|x-y|)}+\frac{|u(x)-u(y)|}{\ell^{\a_1}(|x-y|)}\2\\
&\leq \ell^{\a_1-\a_0}(\delta)2M\leq \e.
\end{align*}
On the other hand, if $|x-y|\geq \d$ we split the quotient as
\begin{align*}
\frac{|(u_{i_j}-u)(x)-(u_{i_j}-u)(y)|}{\ell^{\a_0}(|x-y|)} 
&\leq \ell^{-\a_0}(\d)(|u_{i_j}(x)-u(x)|+|u_{i_j}(y)-u(y)|)\\
&\leq 2\ell^{-\a_0}(\d)\|u_{i_j}-u\|_{L^\infty(\Omega)}.
\end{align*}
Then, for $j$ sufficiently large, independently of $x$ or $y$, we get the right-hand side arbitrarily small with which we conclude the proof.
\end{proof}

\begin{lemma}[Compactness on weighted spaces]\label{cpt:local}
Let $\Omega$, $\mathcal X^\alpha(\Omega)$, and $\mathcal Y(\Omega)$ as in Definition \ref{XY}. For any sequence $(u_i) \ss \mathcal X^\alpha(\Omega)$ such that 
\begin{align}\label{M2}
M&:= \sup_i\|u_i\|_{\mathcal X^\alpha(\Omega)}<\8     
\end{align}
 there exist $u \in \mathcal X^\alpha(\Omega)$ with $\|u\|_{\mathcal X^\alpha(\Omega)}<\8$ such that, passing to a sub-sequence,
 \begin{align*}
\lim\limits_{i\to\infty}\|u_{i}-u\|_{\mathcal Y(\Omega)}=\lim\limits_{i\to\infty}\|u_{i}-u\|_{\cL^{1}(\W)}^{(1)}+\|u_{i}-u\|_{L^\infty(\W)}=0.
 \end{align*}
\end{lemma}
\begin{proof}
By \eqref{M2} we know that
\begin{align*}
M=\sup_i \left(\|u_i\|_{L^\infty(\Omega)} + 
\sup_{\substack{x,y \in \W\\x\neq y}} \ell^{1+\a}(d(x,y))\frac{|u_i(x)-u_i(y)|}{\ell^{1+\a}(|x-y|)} + \sup_{\substack{x,y\in B_1(\W)\\x\neq y}}\frac{|u_i(x)-u_i(y)|}{\ell^\a(|x-y|)}\right)<\8,
\end{align*}
where $d(x,y):=\min(d(x),d(y))$.

There is $n_0\in\mathbb N$ such that 
\begin{align*}
\Omega_n:=\{x\in\Omega \::\: \dist(x,\partial\Omega)\geq \tfrac{1}{n}\}\neq \emptyset\qquad \text{ for all }n\geq n_0.
\end{align*}
Using Lemma \ref{cpt} and a standard diagonalization procedure, there is $u\in \cL^{1+\alpha}_{loc}(\Omega)\cap \cL^\alpha(B_1(\Omega))$ and a subsequence, denoted again by $u_i$, such that 
\begin{align*}
\lim\limits_{i\to\infty}\|u_{i}-u\|_{\cL^{1}(\Omega_n)}
+\|u_{i}-u\|_{L^{\infty}(B_1(\Omega))}=0\qquad \text{ for all }n\in\N.
\end{align*}
Moreover, for $x,y\in\W$ different,
\[
\ell^{1+\a}(d(x,y))\frac{|u(x)-u(y)|}{\ell^{1+\a}(|x-y|)} = \ell^{1+\a}(d(x,y))\lim_{j\to\8}\frac{|u_j(x)-u_j(y)|}{\ell^{1+\a}(|x-y|)} \leq \sup_{j}[u_j]^{(0)}_{\cL^{1+\a}(\W)}\leq M.
\]
Since the right-hand side is bounded independently of $x$ and $y$, we have that  $\|u\|^{(0)}_{\cL^{1+\a}(\W)} \leq M$.  Therefore, $\|u\|_{\mathcal X}<\8$. 

Let $\e>0$ be fixed, our goal is now to bound by $\e$
\[
\limsup_{j\to\8}\sup_{\substack{x,y\in\W\\x\neq y}} \ell^{2}(d(x,y))\frac{|(u_{i}-u)(x)-(u_{i}-u)(y)|}{\ell(|x-y|)}.
\]

We consider two scenarios depending on
\[
\d := \sup\left\{d>0 \::\: \ell^{\a}(d) \leq \frac{\e}{2 M\ell^{1-\a}(\diam(\Omega))}\right\} >0.
\]
If $|x-y|<\d$, we obtain that
\begin{align*}
\ell^{2}(d(x,y))&\frac{|(u_{i}-u)(x)-(u_{i}-u)(y)|}{\ell(|x-y|)} \leq \ell^{2}(d(x,y))\left(\frac{|u_{i}(x)-u_{i}(y)|}{\ell(|x-y|)}+\frac{|u(x)-u(y)|}{\ell(|x-y|)}\right)\\
&\leq \ell^{1-\a}(d(x,y))\ell^{\a}(|x-y|)
\1\ell^{1+\a}(d(x,y))\frac{|u_{i}(x)-u_{i}(y)|}{\ell^{1+\a}(|x-y|)}
+\ell^{1+\a}(d(x,y))\frac{|u(x)-u(y)|}{\ell^{1+\a}(|x-y|)}\2\\
&\leq \ell^{1-\a}(\diam(\Omega))\ell^{\a}(\delta)2M\leq \e.
\end{align*}
On the other hand, if $|x-y|\geq \d$ we split the quotient as
\begin{align*}
\ell^{2}(d(x,y))\frac{|(u_{i}-u)(x)-(u_{i}-u)(y)|}{\ell^{\a}(|x-y|)} 
&\leq \ell^{2}(\diam(\Omega))\ell^{-\a}(\d)(|u_{i}(x)-u(x)|+|u_{i}(y)-u(y)|)\\
&\leq 2\ell^{2}(\diam(\Omega))\ell^{-\a}(\d)\|u_{i}-u\|_{L^\infty(\Omega)}.
\end{align*}
Then, for $i$ sufficiently large, independently of $x$ or $y$, we get the right-hand side arbitrarily small with which we conclude the proof.
\end{proof}

\begin{lemma}\label{lem:10}
Let $\a,\b\geq 0$ and $f \in \cL_{loc}^\a(\W)$ with $\|f\|_{\cL^\a(\W)}^{(\b)}<\8$. Then there exists an family $(f_i)_{\e\in(0,1)}\subset C^\8(\W)$ that converges to $f$ locally uniformly in $\W$ as $\e\to0^+$, and such that $\sup_{\e\in(0,1)}\|f_\e\|^{(\b)}_{\cL^{\a}(\W)}<\8$.
\end{lemma}

\begin{proof}
Given $\e>0$, let $\eta_\e(x) := \e^{-N}\eta(x/\e)$ where $\eta\in C_0^\8(B_1)$ is non-negative, radially symmetric, and such that $\int_{B_1} \eta=1$.

Let for $d>0$
\[
\W_d := \{x\in\W \ | \ \dist(x,\p\W)>d\}
\]
and for $k\in \N\cup\{0\}$,
\[
V_k := \begin{cases}
\W_{1/2} \text{ if } k=0,\\
\W_{1/2^{k+1}}\sm\overline{\W_{1/2^{k-1}}} \text{ if }k\geq 1,
\end{cases}
\]
a sequence of sets that exhausts $\W$. Let $(\xi_k)$ be a smooth partition of unity subordinated to $(V_k)$ such that $\|D\xi_k\|_{L^\8(\R^N)} \leq C2^k$. Let us see how to get the last assertion.

Let
\begin{align*}
W_k &:= \begin{cases}
\W_1 \text{ if } k=0,\\
\W_{7/2^{k+2}}\sm\overline{\W_{5/2^{k+3}}} \text{ if } k\geq 1,
\end{cases}\\
\xi_{k,0} &:= \dist(\cdot,\R^N\sm W_k),\\
\xi_{k,1} &:= \eta_{1/2^{k+3}}\ast \xi_{k,0},\\
\xi_{k,2} &:= \begin{cases}
\xi_{1,0}+\xi_{1,1} \text{ if } k=0,\\
\xi_{1,k-1}+\xi_{1,k}+\xi_{1,k+1} \text{ if } k\geq 1,
\end{cases}\\
\xi_k &:= \begin{cases}
\xi_{k,1}/\xi_{k,2} \text{ in } W_k,\\
0 \text{ on } \R^N\sm\W_k.
\end{cases}
\end{align*}
It is clear that $(\xi_k)$ is a partition of unity subordinated to $(V_k)$. Moreover, for any $i\in\{0,1,2\}$,
\[
2^k\|\xi_{k,i}\|_{L^\8(\R^N)}+[\xi_{k,i}]_{C^{0,1}(\R^N)}\leq C
\]
and also $\inf_{V_k}\xi_{k,2} \geq c2^{-k}$. This implies, thanks to the product rule, that 
\[
\|D\xi_k\|_{L^\8(\R^N)} = \|D\xi_k\|_{L^\8(V_k)} \leq \frac{\|D\xi_{k,1}\|_{L^\8(V_k)}}{\inf_{V_k}\xi_{k,2}} + \frac{\|\xi_{k,1}\|_{L^\8(V_k)}}{\inf_{V_k}\xi_{k,2}^2}\|D\xi_{k,2}\|_{L^\8(V_k)} \leq C2^k.
\]

Finally, let $f_\e := \sum_{k=0}^\8 \xi_k (\eta_{\e/2^{k+3}} \ast f) \in C^\8(\W)$ which converges to $f$ locally uniformly. Let us show that $\sup_{\e}\|f\|_{\cL^\a(\W)}^{(\b)}$ is bounded.

Given $B_{3r}(x) \in \W$ such that $r\in(0,1/4)$, consider either $k=0$ if $d(x)>3/4$ or $k\geq 1$ such that $d(x) \in (3/2^{k+2},3/2^{k+1}]$. This implies that $B_r(x) \ss V_k$. Moreover, the values of $f$ outside of $B_{2r}(x)$ do not contribute for the construction of $f_\e$ over $B_r(x)$.

Hence,
\[
\|f_\e\|_{L^\8(B_r(x))} \leq \sum_{i=-1}^1\|\xi_{k+i} (\eta_{\e/2^{k+i+2}} \ast f)\|_{L^\8(B_r(x))} \leq C\|f\|_{L^\8(B_{2r}(x))} \leq C\ell^{-\b}(r)\|f\|_{\cL^\a(\W)}^{(\b)}.
\]
On the other hand, we get for the oscillation that, for $y \in B_r(x)$,
\begin{align*}
|f_\e(y)-f_\e(x)| &\leq C\1\sum_{i=-1}^1 |(\eta_{\e/2^{k+i+3}}\ast f)(y)-(\eta_{\e/2^{k+i+3}}\ast f)(x)|+2^k\|f\|_{L^\8(B_{2r}(x))}|y-x|\2,\\
&\leq C\ell^{-\a-\b}(r)\|f\|_{\cL^\a(\W)}^{(\b)}\ell^{\a}(|y-x|).
\end{align*}
We just used that $\|D\xi_k\|_{L^\8(\R^N)} \leq C2^k$ for the first inequality. For the second we used the known oscillation for $f$ and that $2^k\sim r^{-1}$ such that, by the semi-homogeneity,
\[
2^k\|f\|_{L^\8(B_{2r}(x))}|y-x| \leq C\|f\|_{\cL^\a(\W)}^{(\b)}\ell^{-\b}(r)\frac{|y-x|}{r} \leq C\|f\|_{\cL^\a(\W)}^{(\b)}\ell^{-\b}(r) \frac{\ell^\a(|y-x|)}{\ell^\a(r)}.
\]
\end{proof}

\subsection{A geometric lemma}

The following result can also be found in \cite[Lemma 5.2]{MR3995092}, which in fact gives a sharper bound with a slightly more technical proof.

\begin{figure}[ht]
    \centering
    \includegraphics[scale=.4]{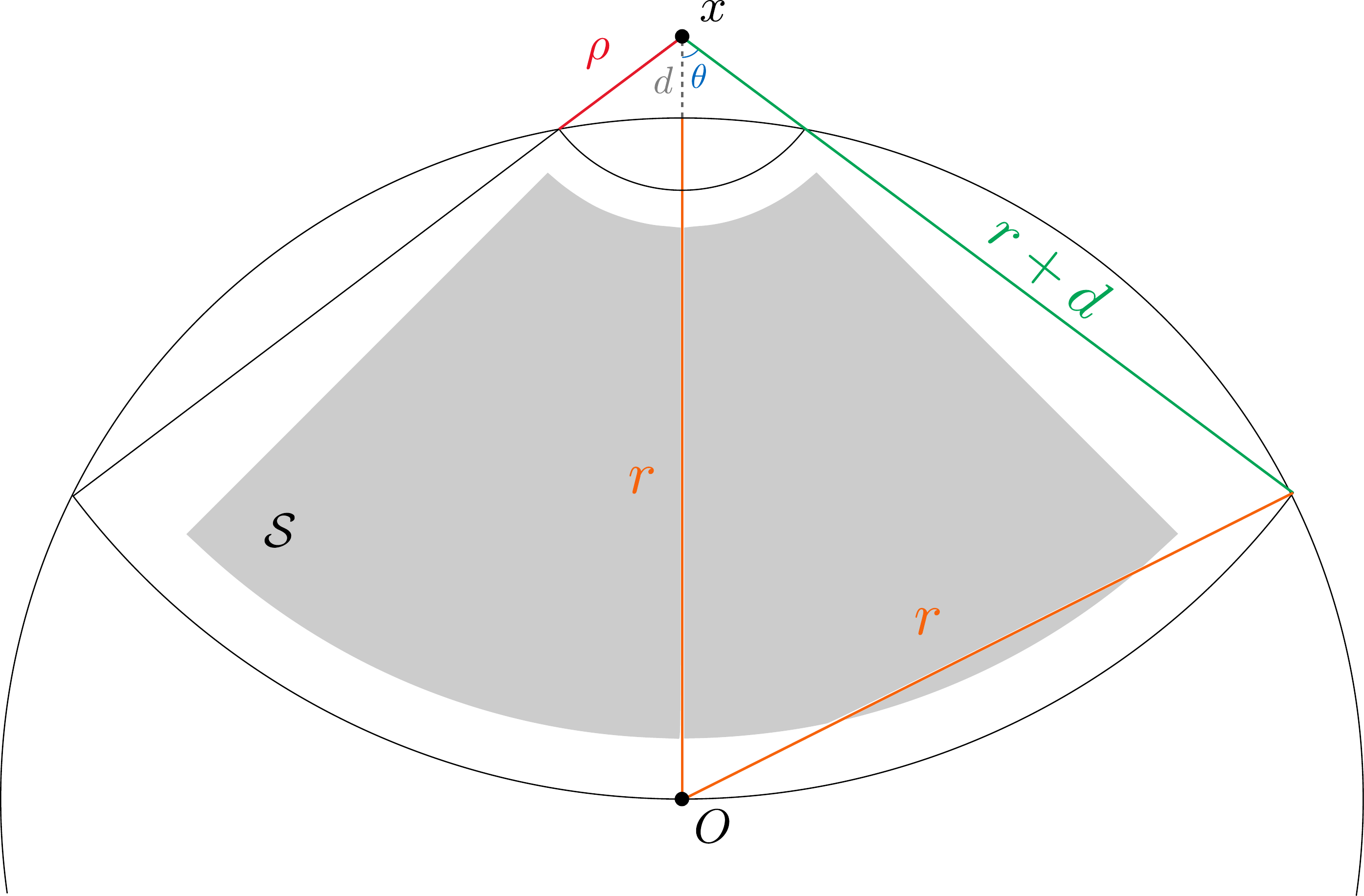}
    \caption{An interior sector.}
    \label{fig:tobias}
\end{figure}

\begin{lemma}\label{lem:tobias}
For $r\in(0,1)$, $x\in B_{r+r^2}\sm B_r$ and $d=|x|-r$
\[
\int_{B_r} \frac{dy}{|y-x|^N} \geq c|\ln d|
\]
for some constant $c>0$ depending only on $N$.
\end{lemma}

\begin{proof}
Under the given restrictions it holds that the following sector is contained in $B_r$
\[
\mathcal S := \{y \in B_{\sqrt d}(x) \sm B_{2d}(x) \ | \ \cos(\angle Oxy) \geq 7/8\}.
\]
This implies the desired bound on the integral once we pass to polar coordinates centered at $x$
\[
\int_{B_r} \frac{dy}{|y-x|^N} \geq \int_{\mathcal S} \frac{dy}{|y-x|^N} = c\int_{2d}^{\sqrt d} \frac{d\r}{\r} = c|\ln(2\sqrt d)| \geq c|\ln d|.
\]

Here is the geometric justification. Consider the cone $\cC$ with vertex at $x$ and generated by $B_r \cap \p B_{r+d}(x)$. For the angle $\theta$ of this cone we estimate the cosine using that $d < r^2<1$
\[
\cos\theta = 1 - \frac{1}{2}\1\frac{r}{r+d}\2^2 \leq 1 - \frac{1}{2}\1\frac{1}{1+r}\2^2 \leq \frac{7}{8}.
\]

Let $y \in \p B_{r+d}(x)\cap \p B_r$ and $z$ the intersection of the segment between $x$ and $y$ and the sphere $\p B_r$.  By a trigonometric argument we get that $z \in \p B_\r(x)\cap \p B_r$ for $\r:=d(2r+d)/(r+d) \leq 2d$. This finally means that $\mathcal S \ss \cC \cap (B_{r+d}(x)\sm B_{\r}(x)) \ss B_r$.
\end{proof}

\end{document}